\documentclass[a4paper,11pt,reqno]{amsart}

\usepackage[utf8]{inputenc}

\usepackage{etex}

\usepackage{xcolor}

\definecolor{verydarkblue}{rgb}{0,0,0.5}

\usepackage[
    breaklinks,
    colorlinks,
    citecolor=verydarkblue,
    linkcolor=verydarkblue,
    urlcolor=verydarkblue,
    pagebackref=true,
    hyperindex
]{hyperref}

\backrefenglish

\usepackage{fancyhdr}

\usepackage[
    hscale=0.7,
    vscale=0.75,
    headheight=13pt,
    centering,
]{geometry}

\usepackage{amsmath}
\usepackage{amsthm}
\usepackage{amssymb}
\usepackage{mathtools}
\usepackage{stmaryrd}
\usepackage{mathdots}
\usepackage{framed}
\usepackage[capitalize]{cleveref}
\usepackage{array}
\usepackage[alphabetic]{amsrefs}
\usepackage[all,cmtip]{xy}

\theoremstyle{plain}

\newtheorem{introtheorem}{Theorem}

\crefname{introtheorem}{Theorem}{Theorems}

\newtheorem{theorem}{Theorem}[section]
\newtheorem{proposition}[theorem]{Proposition}
\newtheorem{lemma}[theorem]{Lemma}
\newtheorem{corollary}[theorem]{Corollary}

\theoremstyle{definition}

\newtheorem{definition}[theorem]{Definition}

\theoremstyle{remark}

\newtheorem{remark}[theorem]{Remark}
\newtheorem{example}[theorem]{Example}

\numberwithin{figure}{section}

\numberwithin{equation}{section}

\IfFileExists{./article-style.tex}{

\pagestyle{fancy}
\fancyhead{}
\fancyfoot{}
\fancyhead[LE,RO]{\small \thepage}
\fancyhead[RE]{\small \nouppercase{\rightmark}}
\fancyhead[LO]{\small \nouppercase{\leftmark}}

\setcounter{tocdepth}{1}

\makeatletter

\global\BR@BackrefAlttrue

\renewcommand*{\backrefalt}[4]{%
    \tiny%
    (%
    \ifcase #1 not cited%
          \or cit.~on~p.~#2%
          \else cit.~on~pp.~#2%
    \fi%
    )%
}

\def\print@backrefs#1{%
    \space\SentenceSpace%
    \begingroup%
        \expandafter\providecommand\csname brc@#1\endcsname{0}%
        \expandafter\providecommand\csname brcd@#1\endcsname{0}%
        \expandafter\backrefalt%
            \csname brc@#1\expandafter\endcsname%
            \csname brl@#1\expandafter\endcsname%
            \csname brcd@#1\expandafter\endcsname%
            \csname brld@#1\endcsname%
    \endgroup%
}

\def\maketitle{\par
  \@topnum\z@ 
  \@setcopyright
  \thispagestyle{empty}
  \ifx\@empty\shortauthors \let\shortauthors\shorttitle
  \else \andify\shortauthors
  \fi
  \@maketitle@hook
  \begingroup
  \@maketitle
  \toks@\@xp{\shortauthors}\@temptokena\@xp{\shorttitle}%
  \toks4{\def\\{ \ignorespaces}}
  \edef\@tempa{%
    \@nx\markboth{\the\toks4
      \@nx\MakeUppercase{\the\toks@}}{\the\@temptokena}}%
  \@tempa
  \endgroup
  \c@footnote\z@
    \renewcommand{\footnoterule}{%
      \kern -3pt
      \hrule width \textwidth height .5pt
      \kern 2pt
    }
  {
    \renewcommand\thefootnote{}
    \vspace{-2em}
    \footnote{
      \par\vspace{-1.2em}\noindent%
      \setlength{\parindent}{0pt}%
      \def\@footnotetext##1{\noindent{\footnotesize##1}\par}%
      \let\@makefnmark\relax  \let\@thefnmark\relax
      \ifx\@empty\@date\else \@footnotetext{\@setdate}\fi
      \ifx\@empty\@subjclass\else \@footnotetext{\@setsubjclass}\fi
      \ifx\@empty\@keywords\else \@footnotetext{\@setkeywords}\fi
      \ifx\@empty\thankses\else \@footnotetext{%
        \@setthanks}%
      \fi
    }
    \addtocounter{footnote}{-1}
  }
  \@cleartopmattertags
}

\def\@adminfootnotes{\@empty}

\def\@settitle{\begin{center}%
  \baselineskip14\p@\relax
    \bfseries
\Large
  \@title
  \end{center}%
}

\def\@setauthors{%
  \begingroup
  \def\thanks{\protect\thanks@warning}%
  \trivlist
  \centering\footnotesize \@topsep30\p@\relax
  \advance\@topsep by -\baselineskip
  \item\relax
  \author@andify\authors
  \def\\{\protect\linebreak}%
  \large{\authors}%
  \ifx\@empty\contribs
  \else
    ,\penalty-3 \space \@setcontribs
    \@closetoccontribs
  \fi
  \endtrivlist
  \endgroup
}

\def\@setaddresses{\par
  \nobreak \begingroup
\footnotesize
  \def\author##1{\end{minipage}\hskip 1sp \begin{minipage}{.5\textwidth}\raggedright%
    ~\\[2em]{\bf##1}\\[.5em]%
  }%
  \interlinepenalty\@M
  \def\address##1##2{\begingroup
    {\ignorespaces##2}\endgroup\\[.5em]}%
  \def\curraddr##1##2{\begingroup
    \@ifnotempty{##2}{\nobreak\indent\curraddrname
      \@ifnotempty{##1}{, \ignorespaces##1\unskip}\/:\space
      ##2\par}\endgroup}%
  \def\email##1##2{\begingroup
    \@ifnotempty{##2}{\nobreak\indent
      \@ifnotempty{##1}{, \ignorespaces##1\unskip}
      \ttfamily##2\par}\endgroup}%
  \def\urladdr##1##2{\begingroup
    \def~{\char`\~}%
    \@ifnotempty{##2}{\nobreak\indent\urladdrname
      \@ifnotempty{##1}{, \ignorespaces##1\unskip}\/:\space
      \ttfamily##2\par}\endgroup}%
  \setlength{\parindent}{0pt}%
  \vfill%
  {
  \begin{minipage}{0mm}
  \addresses
  \end{minipage}
  }
  \endgroup
}

\renewcommand{\author}[2][]{%
  \ifx\@empty\authors
    \gdef\authors{#2}%
    \g@addto@macro\addresses{\author{#2}}%
  \else
    \g@addto@macro\authors{\and#2}%
    \g@addto@macro\addresses{\author{#2}}%
  \fi
  \@ifnotempty{#1}{%
    \ifx\@empty\shortauthors
      \gdef\shortauthors{#1}%
    \else
      \g@addto@macro\shortauthors{\and#1}%
    \fi
  }%
}
\edef\author{\@nx\@dblarg
  \@xp\@nx\csname\string\author\endcsname}

\def\@secnumfont{\@empty}

\def\section{\@startsection{section}{1}%
  \z@{.7\linespacing\@plus\linespacing}{.5\linespacing}%
  {\large\bfseries\centering}}

\makeatother

}{}

\renewcommand{\llbracket}{[\hspace{-0.5pt}[}
\renewcommand{\rrbracket}{]\hspace{-0.5pt}]}

\newcommand{\Cont}{{\rm Cont}}
\newcommand{\red}{{\rm red}}

\newcommand{\coker}{\operatorname{coker}}

\newcommand{\jetcodim}{\,{\rm jet.codim}}
\newcommand{\embdim}{\,{\rm emb.dim}}

\newcommand{\trdeg}{\,{\rm tr.deg}}
\newcommand{\ord}{\operatorname{ord}}

\newcommand{\Jac}{\operatorname{Jac}}
\newcommand{\Hom}{\operatorname{Hom}}
\newcommand{\Der}{\operatorname{Der}}
\newcommand{\Spec}{\operatorname{Spec}}
\newcommand{\Spf}{\operatorname{Spf}}
\newcommand{\Fitt}{\operatorname{Fitt}}
\newcommand{\Sing}{\operatorname{Sing}}
\renewcommand{\Im}{\operatorname{Im}}
\newcommand{\HS}{{\rm HS}}

\begin{document}

\title{Differentials on the arc space}

\author{Tommaso de Fernex}

\address[T.\ de Fernex]{%
    Department of Mathematics\\
    University of Utah\\
    155 South 1400 East\\
    Salt Lake City, UT 48112 (USA)%
}

\email{defernex@math.utah.edu}

\author{Roi Docampo}

\address[Roi Docampo]{%
    Department of Mathematics\\
    University of Oklahoma\\
    601 Elm Avenue, Room 423\\
    Norman, OK 73019 (USA)%
}

\email{roi@ou.edu}

\subjclass[2010]{Primary 14E18; Secondary 13N05, 13H99.}
\keywords{Arc spaces, jet schemes, K\"ahler differentials, embedding dimension}

\thanks{The research of the first author was partially supported by NSF Grants
DMS-1402907 and DMS-1700769, and NSF FRG Grant DMS-1265285.}

\begin{abstract}
The paper provides a description of the sheaves of K\"ahler differentials of
the arc space and jet schemes of an arbitrary scheme where these sheaves 
are computed directly from the sheaf
of differentials of the given scheme. Several applications on the structure of
arc spaces are presented.
\end{abstract}

\maketitle

\section{Introduction}

The work of Greenberg, Nash, Kolchin, and Denef--Loeser
\cite{Gre66,Nas95,Kol73,DL99} has set the basis for our understanding of the
structure of arc spaces and their connections to singularities and birational
geometry. Most of the focus in these studies is on the reduced structure of arc
spaces and their underlying topological spaces, and little is known about their
scheme structure. Notable studies in this direction are those of Reguera
\cite{Reg06,Reg09}, recently continued in \cite{Reg,MR}. 

This paper can be viewed as a continuation of these studies. In the first part
we describe the sheaves of K\"ahler differentials of the arc space and of the
jet schemes.
The second part of the paper is devoted to applications of the
resulting formulas. The approach leads to new results as well as simpler and
more direct proofs of some of the theorems in the literature, and provides a
new way of understanding some of the fundamental properties of the theory.  

In the first part of the paper we work over an arbitrary base scheme $S$.
For simplicity, in the introduction we restrict to the case of
a scheme $X$ over a field $k$; 
the reader will not loose too much of the spirit of the paper
by even assuming that $X$ is a variety.
The \emph{arc space} of $X$ over $S$, denoted
$X_\infty$ parametrizes formal arcs on $X$, 
and comes equipped with a universal family known as the \emph{universal arc}:
\[
\xymatrix{
    U_\infty \ar[r]^{\gamma_\infty} \ar[d]_{\rho_\infty} &
    X \\
    X_\infty  & 
}
\]
More concretely, if $X \subset \mathbb A^n$ is an affine scheme
defined by polynomial equations $f_j = 0$, 
then an arc on $X$ is a vector of power series $(x_1(t),\dots,x_n(t))$ 
with coefficients in a field satisfying identically the equations $f_j = 0$. 
The coefficients of the power series define coordinates on $X_\infty$. 
The universal arc is a vector of power series $(x_1(t),\dots,x_n(t))$
whose coefficients are the coordinate functions on $X_\infty$.


%

As a word of warning, $X_\infty$ and $U_\infty$ are typically not Noetherian,
even when $X$ is.
On the other hand, they are affine over $X$ and this gives us
concrete tools to work with them. 
We note that all the theorems in this paper are local in nature, and thus
one can restrict without loss of generality to the case
where $X$ is affine. 

On $U_\infty$, we construct a sheaf $\mathcal
P_\infty$ whose $\mathcal O_{X_\infty}$-dual is $\mathcal O_{U_\infty}$.
The sheaf $\mathcal P_\infty$ plays the role of ``kernel''
(as in a sort of Fourier--Mukai transform) in the next
formula which relates the sheaf of differentials of $X_\infty$ directly to
the sheaf of differentials of $X$. 

\begin{introtheorem}[K\"ahler differentials on arc spaces]
\label{th:differentials-arc-space}
There is a natural isomorphism
\[
    \Omega_{X_\infty} 
    \simeq 
    \rho_{\infty*}(\gamma_\infty^*(\Omega_{X}) 
    \otimes
    \mathcal P_\infty).
\]
\end{introtheorem}

We have a similar result for jet schemes. If $X_n$ denotes the $n$-th
\emph{jet scheme} of $X$, then we have a \emph{universal jet} $X_n
\xleftarrow{\rho_n} U_n \xrightarrow{\gamma_n} X$, and we construct a sheaf
$\mathcal P_n$ on $U_n$ whose $\mathcal O_{X_n}$-dual is $\mathcal
O_{U_n}$. In this case, it turns out that $\mathcal P_n$ is the $\mathcal
O_{X_n}$-dual of $\mathcal O_{U_n}$ and is trivial as an $\mathcal
O_{U_n}$-module, and the description of the sheaf of differentials gets
simplified.

\begin{introtheorem}[K\"ahler differentials on jet schemes]
\label{th:differentials-jet-scheme}
There are natural isomorphisms
\[
    \Omega_{X_n} 
    \simeq 
    \rho_{n*}(\gamma_n^*(\Omega_{X}) 
    \otimes
    \mathcal P_n)
    \simeq 
    \rho_{n*}(\gamma_n^*(\Omega_{X})).
\]
\end{introtheorem}

These theorems provide an efficient way of computing 
the Jacobian matrices of the jet schemes and the arc space. 
We illustrate this in some details at the end of \cref{s:der+diff}.

We apply \cref{th:differentials-arc-space,th:differentials-jet-scheme} 
to study the structure of arc spaces.

The starting point is the description of the fibers of the sheaf of
differentials of the jet schemes $X_n$ at liftable jets. Suppose that $X$
is a scheme of finite type over $k$, and let $\alpha \in X_\infty$ be an
arc. For $n \gg 1$, let $\alpha_n \in X_n$ be the truncation of $\alpha$
and $L_n$ the residue field of $\alpha_n$. Using the above theorems, we
determine an isomorphism
\[
    \Omega_{X_n} \otimes L_n
    \;\simeq\;
    \big(L_n[t]/(t^{n+1})\big)^d
    \;\oplus\;
    \bigoplus_{i \geq d} L_n[t]/(t^{e_i}).
\]
Here the number $d$ and the sequence $\{e_i\}$ are certain Fitting-theoretic
invariants of the pull-back of $\Omega_{X}$ by $\alpha$ (see \cref{s:fitting}
for the precise definition of these invariants and \cref{th:fitting} for the
precise statement). If $X$ is a reduced and equidimensional scheme of finite
type over a field $k$, and $\alpha$ is not fully contained in the singular locus of $X$, then
we have $d = \dim X$ and $\sum_{i \ge d} e_i =
\ord_\alpha(\Jac_X)$ where $\Jac_X$ is the Jacobian ideal of $X$, and the above
isomorphism recovers in this case one of the main results of \cite{dFD14}.

These results have several applications. For simplicity, for the reminder of
the introduction we shall assume that $X$ is a variety defined over a perfect
field $k$, though more general results are obtained in the paper. 

A natural way of studying the structure of the arc space of $X$ is to analyze
its (not necessarily closed) points. Given a point $\alpha \in X_\infty$, we
are interested in two invariants: the \emph{embedding dimension}
\[
\embdim(\mathcal O_{X_\infty,\alpha})
\]
of the local ring at $\alpha$, and the \emph{jet codimension} of
$\alpha$ in $X_\infty$, which is defined by 
\[
	\jetcodim(\alpha,X_\infty)
	:=
        \lim_{n \to \infty}\big((n+1)\dim (X)
        - \dim (\overline{\{\alpha_n\}})\big).
\]
These and related invariants have been studied in the literature (e.g., see
\cite{ELM04,dFEI08,dFM15,Reg,MR}). Both numbers provide measures of the
``size'' of the point. Note however that while the embedding dimension is
computed on the arc space (with its scheme structure), the jet
codimension is computed from the truncations of the arc and only
depends on the reduced structure of the jet schemes. 

Our first application is that these
two invariants measure the same quantity. 

\begin{introtheorem}[Embedding dimension as jet codimension]
\label{th:embdim=cylcodim-finiteness-intro}
Given a variety $X$ over a perfect field, we have
\[
	\embdim(\mathcal O_{X_\infty,\alpha}) = \jetcodim(\alpha,X_\infty)
\]
for every $\alpha \in X_\infty$. 
\end{introtheorem}

One of the most important results on arc spaces is the Birational
Transformation Rule \cite{Kon95,DL99}. It implies the change-of-variable
formula in motivic integration \cite{Kon95,DL99,Bat99,Loo02} and has been
applied to study invariants of singularities in birational geometry (e.g., see
\cite{Mus01,Mus02,EMY03,EM04,ISW12,Ish13,IR13,dFD14,EI15,Zhu15}). Using our
description of the sheaf of differentials, we obtain the following variant.

\begin{introtheorem}[Birational transformation rule]
\label{th:BTR-intro}
Given a proper birational map $f \colon Y \to X$ between two varieties over a
perfect field, we have
\[
    \embdim\left(\mathcal O_{Y_\infty,\beta}\right)
    \leq
    \embdim\left(\mathcal O_{X_\infty,f_\infty(\beta)}\right)
    \leq
    \embdim\left(\mathcal O_{Y_\infty,\beta}\right)
    + \ord_{\beta}(\Jac_f)
\]
for every $\beta \in Y_\infty$, where $\Jac_f$ is the Jacobian of $f$. 
Moreover, if $Y$ is smooth at $\beta(0)$, then 
\[
    \embdim\left(\mathcal O_{X_\infty,f_\infty(\beta)}\right)
    =
    \embdim\left(\mathcal O_{Y_\infty,\beta}\right)
    + \ord_{\beta}(\Jac_f).
\] 
\end{introtheorem}

The connection between this theorem and Denef--Loeser's Birational
Transformation Rule becomes evident once one rewrites the second formula of
\cref{th:BTR-intro} using \cref{th:embdim=cylcodim-finiteness-intro}, which
gives the formula
\[
    \jetcodim(f_\infty(\beta),X_\infty)
    =
    \jetcodim(\beta,Y_\infty)
    + \ord_{\beta}(\Jac_f)
\] 
for any resolution of singularities $f \colon Y \to X$. Although it does not
retain all the information provided in \cite{DL99} that is necessary for the
change-of-variable formula in motivic integration, this formula suffices for
all known applications to the study of singularities in birational geometry. 

Our next application regards the \emph{stable points} of the arc space. 
These are the generic points of the irreducible constructible
subsets of $X_\infty$ that are not contained in $(\Sing X)_\infty$
(see \cref{s:stable-points}). Stable points and their local rings
have been extensively studied in \cite{Reg06,Reg09,Reg,MR}. The following
theorem can be viewed as providing a characterization of stable points. 

\begin{introtheorem}[Characterization of finite embedding dimension]
\label{th:char-stable-points-intro}
Given a variety $X$ over a perfect field and a point $\alpha \in X_\infty$, we have
\[
	\embdim(\mathcal O_{X_\infty,\alpha}) < \infty
\]
if and only if $\alpha$ is a stable point.
\end{introtheorem}

It follows from general properties of local rings
that the embedding dimension at a point $\alpha
\in X_\infty$ is finite if and only if 
the completion of the local ring is Noetherian. 
The fact that the completion of the local
ring at a stable point $\alpha \in X_\infty$ 
is Noetherian is a theorem of Reguera \cite{Reg06,Reg09}, 
and we get a new proof of this important result.
It is the key ingredient in the proof of the Curve Selection Lemma, which plays an
essential role in the recent progress on the Nash problem (e.g., see
\cite{FdBPP12,dFD16}).

There are examples in positive characteristics of varieties $X$ whose arc space
$X_\infty$ has irreducible components that are fully contained in $(\Sing
X)_\infty$, and \cref{th:char-stable-points-intro} implies that $X_\infty$ has
infinite embedding codimension at the generic points of such components (see
\cref{th:min-primes-inf-embdim}). One should contrast this with the main theorem
of \cite{GK00,Dri}, which can be interpreted as 
saying that the completion of the local ring of $X_\infty$ at any $k$-valued
point that is not contained in $(\Sing X)_\infty$ has finite embedding codimension. 

A special class of stable points is given by what we call the \emph{maximal
divisorial arcs}. By definition, these are the arcs $\alpha \in X_\infty$ whose
associated valuation $\ord_\alpha$ is a divisorial valuation and that are
maximal (with respect to specialization) among all arcs defining the same
divisorial valuation. Equivalently, they are the generic points of the maximal
divisorial sets defined in \cite{ELM04,dFEI08,Ish08}. 
For example, if $E$ is a prime divisor on a resolution of singularities $f \colon Y \to X$, 
and $C \subset Y_\infty$ is the set of arcs on $Y$ with positive order of contact
along $E$, then the closure of $f_\infty(C)$ in $X_\infty$ is the maximal
divisorial set associated to the valuation $\ord_E$ and
its generic point is the maximal divisorial arc corresponding to this valuation. 
Our final application gives the following result. 

\begin{introtheorem}[Embedding dimension at maximal divisorial arcs]
\label{th:max-div-arc-embdim-intro}
Let $X$ be a variety over a perfect field, 
$f \colon Y \to X$ a proper birational morphism from a normal variety $Y$,
$E$ a prime divisor on $Y$, and $q$ a positive integer. 
If $\alpha \in X_\infty$ is the maximal divisorial arc 
corresponding to the divisorial valuation $q\ord_E$, then
\[
	\embdim(\mathcal O_{X_\infty,\alpha}) 
	= q(\ord_E(\Jac_f) + 1).
\]
\end{introtheorem}

The quantity $\ord_E(\Jac_f)$ is also known as the \emph{Mather discrepancy} of $E$ over $X$,
and denoted by $\widehat k_E(X)$ \cite{dFEI08}. 
In view of \cref{th:embdim=cylcodim-finiteness-intro},
\cref{th:max-div-arc-embdim-intro} recovers \cite[Theorem~3.8]{dFEI08}. 
The theorem is also closely related to
a recent result of Mourtada and Reguera \cite{Reg,MR}
which states that with the same assumptions as in \cref{th:max-div-arc-embdim-intro}, 
if the field has characteristic zero then 
\[
	\embdim(\widehat{\mathcal O_{X_\infty,\alpha}}) 
	= \embdim(\mathcal O_{(X_\infty)_\red,\alpha}) 
	= q(\ord_E(\Jac_f) + 1).
\]
Here, $\widehat{\mathcal O_{X_\infty,\alpha}}$ is the completion of 
$\mathcal O_{X_\infty,\alpha}$ with respect to the $I$-adic topology where
$I \subset \mathcal O_{X_\infty,\alpha}$ is the maximal ideal. 
It is regarded with the inverse limit topology, which in general
differs from the $\widehat I$-topology (a system of neighborhood of $0$
is given by the closures of the powers of $\widehat I$, which can be
strictly larger than the powers themselves).
It can be shown using results from \cite{Reg09} that, in characteristic zero,
\cref{th:max-div-arc-embdim-intro} also follows from the above
theorem of Mourtada and Reguera (see \cref{r:MR-Reg09}).

In a different direction, the isomorphisms given in \cref{th:differentials-jet-scheme}
can be used to study the relationship between the Nash blow-up
of a variety and the Nash blow-up of its jet schemes.
This study has been be carried out in the forthcoming paper \cite{dFD}.

Proofs of the results stated in the introduction are located as follows: both
statements in \cref{th:differentials-arc-space,th:differentials-jet-scheme} are
contained in \cref{th:differentials},
\cref{th:embdim=cylcodim-finiteness-intro} is proved in 
\cref{th:embdim=cylcodim}, \cref{th:BTR-intro} combines the
statements of \cref{th:BTR,th:BTR-gen}, \cref{th:char-stable-points-intro}
is proved in \cref{th:char-finite-cylcodim}, and
\cref{th:max-div-arc-embdim-intro} follows from \cref{th:max-div-arc-embdim}.

It is worthwhile mentioning that 
most proofs in this paper rely only on the definition of arc space and basic
facts in commutative algebra.

\subsection*{Acknowledgments}

The problem of understanding the tangent sheaf of the arc space
was proposed by Lawrence Ein to the second author when he was his graduate student, 
for which we are very grateful. 
We would like to thank Hussein Mourtada and Ana Reguera for sending us a copy
of their preprint \cite{MR} before it was available online.
The main result of their paper served as an inspiration which led us 
to the statement of \cref{th:embdim=cylcodim-finiteness-intro}, 
though through a very different path.
We thank Ana Reguera for useful comments.
Finally, we would like to thank the referees
for their careful reading of the paper and many
valuable comments and suggestions which have helped 
us improve the paper. 
In particular, we are grateful to one of the referees for 
pointing out an error in a previous version of \cref{th:lambda},
bringing to our attention the property stated in \cref{th:stable-separable-extensions}, 
and suggesting \cref{r:MR-Reg09}.

\section{Conventions}
\label{s:conventions}

Throughout the paper, all rings are assumed to be commutative with identity.
Unless otherwise specified, rings are regarded with the discrete topology;
however, power series rings of the form $R\llbracket t \rrbracket$ are
considered as complete topological rings. For topological modules $M$ and $N$
over a topological ring $R$, we define their completed tensor product, denoted
by $M \operatorname{\hat\otimes_R} N$, as the completion of the ordinary tensor
product $M \otimes_R N$. We will mostly encounter completed tensor products of
the form $M \operatorname{\hat\otimes_R} A\llbracket t \rrbracket$, where $R$
is a ring, $A$ is an $R$-algebra, and $M$ is an $R$-module, all with the
discrete topology.

We fix a base scheme $S$ and work on the category of schemes over $S$. Given an
object $X$ in this category, we do not impose any condition on the morphism $X
\to S$. However, starting with \cref{s:embdim} we will assume that $S = \Spec
k$ where $k$ is a perfect field, and mostly focus on schemes of finite type
over $k$.

We also need to consider formal schemes over $S$. For our purposes, it will be
enough to consider the notion of formal scheme introduced in \cite{EGAi}. In
fact, we will only consider formal schemes of the form $X
\operatorname{\hat\times_{\mathbb Z}}\Spf \mathbb Z\llbracket t\rrbracket$
where $X$ is an ordinary scheme. Here and in the sequel we use the symbol
$\hat\times$ to denote the product in the category of formal schemes; this
emphasizes the fact that it corresponds to the completed tensor product
$\hat\otimes$ at the level of topological rings. In particular,
\[
    \Spec R 
    \operatorname{\hat\times_{\mathbb Z}}
    \Spf \mathbb Z\llbracket t\rrbracket
    =
    \Spf (R \operatorname{\hat\otimes_{\mathbb Z}} \mathbb Z\llbracket t\rrbracket)
    =
    \Spf R\llbracket t\rrbracket.
\]

Unless otherwise stated, the letters $m$ and $n$ will be used to denote
elements in the set $\mathbb N \cup \{\infty\}$ where $\mathbb N$ denotes the set
of nonnegative integers.

\section{Generalities on arcs and jets}

\label{s:arcs-jets}

Let $X$ be an arbitrary scheme over a base scheme $S$. The jet schemes and the
arc space of $X$ over $S$ are defined, in this generality, in \cite{Voj07}, to
which we refer for more details and proofs; see also \cite{IK03,EM09}. 

For every non-negative integer $n$, the \emph{$n$-jet scheme} $(X/S)_n$ of $X$
over $S$ represents the functor from $S$-schemes to sets given by 
\[
    Z 
    \mapsto 
    \Hom_S(Z\times_\mathbb Z \Spec \mathbb Z[t]/(t^{n+1}),X),
\]
while the \emph{arc space} $(X/S)_\infty$ of $X$ over $S$ represents the
functor from $S$-schemes to sets given by 
\[
    Z
    \mapsto
    \Hom_S(
        Z 
        \operatorname{\hat\times_\mathbb Z} 
        \Spf \mathbb Z\llbracket t\rrbracket,
        X
    ).
\]
A point of $(X/S)_n$ is called an \emph{$n$-jet} of $X$ (over $S$), and a point of
$(X/S)_\infty$ an \emph{arc} of $X$ (over $S$). 

Note that an arc $\alpha \in (X/S)_\infty$
can be equivalently thought as a map $\Spf L\llbracket t \rrbracket \to X$ or a
map $\Spec L\llbracket t \rrbracket \to X$ where $L$
is the residue field of $\alpha$
and the composition of the map with the structure map $X \to S$
factors through $\Spec L$, see \cref{th:affine-univ-prop} below. 
In particular, if $S = \Spec R$ then we can view $\alpha$ as an arc with coefficients in $R$
via the map $R \to L$. 
The advantage of considering $\alpha$ as a map $\Spec L\llbracket t \rrbracket \to
X$ is that it allows us to talk about the \emph{generic point of the arc}, by
which we mean the image $\alpha(\eta) \in X$ of the generic point $\eta$ of
$\Spec L\llbracket t \rrbracket$. We will denote by $\alpha(0)$ the image of
the closed point of $\Spec L\llbracket t \rrbracket$.

If $S = \Spec R$ where $R$ is a ring, then one can replace $\mathbb Z$ with $R$
in the above formulas. If $S = \Spec k$ where $k$ is a field, then we will
simply denote $(X/S)_n$ by $X_n$ and $(X/S)_\infty$ by $X_\infty$. 

For any $n \in \mathbb N \cup \{\infty\}$, the scheme $(X/S)_n$ is equipped with
a universal family
\begin{equation}
\label{eq:univ-jet}
    \xymatrix{
        U_n \ar[r]^{\gamma_n} \ar[d]_{\rho_n} &
        X \ar[d] \\
        (X/S)_n \ar[r] & S
    }
\end{equation}
For $n$ finite, the family is given by 
\[
    U_n
    =
    (X/S)_n \times_\mathbb Z \Spec \mathbb Z[t]/(t^{n+1}), 
\]
and is called the \emph{universal $n$-jet}. For $n = \infty$, it is given by
\[ 
    U_\infty
    =
    (X/S)_\infty \hat\times_\mathbb Z \Spf \mathbb Z\llbracket t\rrbracket
\]
and is called the \emph{universal arc}. Notice that $U_\infty$ is a formal
scheme.
The completed fiber 
of $\rho_\infty \colon U_\infty \to (X/S)_\infty$
over a point $\alpha \in (X/S)_\infty(L)$ is $\Spf L\llbracket t \rrbracket$, and
the restriction of $\gamma_\infty$ to it agrees
with the induced map $\alpha \colon \Spf L\llbracket t \rrbracket \to X$. 
While the definition of $U_\infty$ may be hard to grasp at first sight, 
we will see in a moment that in the affine case its structure can be
described very explicitly.

We will use the following notations for the natural truncation maps:
\[
    \pi_n \colon (X/S)_\infty \to (X/S)_n,
    \quad
    \pi_{m,n} \colon (X/S)_m \to (X/S)_n,
    \quad
    \psi_n \colon (X/S)_n \to X.
\]
Note that $\gamma_n$ is different from the composition map $\psi_n \circ \rho_n$. 
Furthermore, observe that there is no natural map
between $U_m$ and $U_n$ when $m > n$. The natural map is
\[
    \mu_{m,n} \colon U_n \times_{(X/S)_n} (X/S)_m \longrightarrow U_m,
\]
where the fiber product is taken with respect to the maps $\rho_n$
and $\pi_{m,n}$.

It will be useful to have notation in place for the affine case. The following
basic properties are proved in \cite[Corollary~1.8 and Theorem~4.5]{Voj07}.

\begin{lemma}
\label{th:affine-univ-prop}
Assume that $S = \Spec R$ for a ring $R$ and $X = \Spec A$ for an $R$-algebra
$A$. Then $(X/S)_n$ is affine for all $n \in \mathbb N \cup \{\infty\}$. 
If we write $(X/S)_n = \Spec A_n$,
then for every $R$-algebra $C$ we have
\[
    (X/S)_n(C) 
    =
    \Hom_{\text{$R$-alg}}(A_n, C)
    = 
    \Hom_{\text{$R$-alg}}(A,C[t]/(t^{n+1})), 
\]
if $n$ is finite, and
\[
    (X/S)_\infty(C) 
    =
    \Hom_{\text{$R$-alg}}(A_\infty, C)
    = 
    \Hom_{\text{$R$-alg}}(A,C\llbracket t \rrbracket), 
\]
Moreover, $A_n$ is characterized by the above property.
\end{lemma}

More explicitly, jet schemes and arc spaces can be defined using Hasse--Schmidt
derivations.
With the notation of \cref{th:affine-univ-prop}, 
the $A$-algebra $A_n$ can be constructed as the 
algebra of Hasse--Schmidt differentials $\HS^n_{A/R}$,
see \cite[Definition~1.3 and Theorem~4.5]{Voj07}.
It comes equipped with the universal Hasse--Schmidt derivation, 
which is a sequence $(D_0,D_1,\dots,D_n)$ 
(by which we mean $(D_0,D_1,\dots)$ if $n = \infty$)
where $D_0 \colon A \to \HS^n_{A/R}$ is the natural inclusion
and $D_i \colon A \to \HS^n_{A/R}$, for $i \ge 1$,
are group homomorphisms such that $D_i(r) = 0$ for $r \in R$
and 
\[
D_i(xy) = \sum_{j+k=i} D_j(x)D_k(y)
\]
for all $x,y \in A$.

\begin{remark}
The definition of Hasse--Schmidt derivations is tailored to work 
in arbitrary characteristic. In characteristic zero
Hasse--Schmidt derivations can be computed from usual derivations
(see \cite[Section~1, Example]{Voj07}) but this is no longer true 
if the characteristic is positive.
\end{remark}

Defining $B_n := A_n[t]/(t^{n+1})$ when $n$ is
finite and $B_\infty := A_\infty\llbracket t\rrbracket$ when $n = \infty$, we
have
\[
    (X/S)_n = \Spec A_n
    \quad 
    \text{and} 
    \quad
    U_n = \Spf B_n.
\]
The universal jet (or arc) is given by
\begin{equation}
\label{eq:univ-jet-rings}
    \xymatrix{
        B_n &
        A \ar[l]_-{\gamma_n^{\sharp}} \\
        A_n \ar[u]^{\rho_n^\sharp} & R \ar[u]\ar[l]
    }
\end{equation}
where the map $\rho_n^\sharp$ is the natural inclusion $A_n \subset B_n$, and
the map $\gamma_n^\sharp$ is defined by
\[
    \gamma_n^{\sharp}(f) = \sum_{p=0}^n D_p(f) \, t^p
\]
where $(D_p)_{p=0}^n$ is the universal Hasse--Schmidt derivation.

We will consider in $B_n$ the $A$-module structure given by $\gamma_n^\sharp$
and the $A_n$-module structure given by $\rho_n^\sharp$. Notice that $B_n$ has
a second $A$-module structure (induced from the inclusion $A \subset A_n
\subset B_n$), but we will have no use for it.

For $m > n$, the map $\mu_{m,n} \colon U_n \times_{(X/S)_n} (X/S)_m \to U_m$
is defined by the natural projection
\[
    \mu^\sharp_{m,n}
    \colon
    B_m
    = A_m[t]/(t^{m+1})
    \longrightarrow
    B_n \otimes_{A_n} A_m
    = A_m[t]/(t^{n+1})
\]
when $m$ is finite, and
\[
    \mu^\sharp_{\infty,n}
    \colon
    B_\infty
    = A_\infty\llbracket t \rrbracket
    \longrightarrow
    B_n \otimes_{A_n} A_\infty
    = A_\infty[t]/(t^{n+1})
\]
when $m = \infty$.

\begin{remark}
If $X$ is a quasi-compact and quasi-separated scheme over a field $k$, then it
follows from the results in \cite{Bha16} that the functor of points of the arc
space $X_\infty$ can also be described as
\[
    X_\infty(C) 
    = 
    \Hom_k(
        \Spec C\llbracket t \rrbracket,
        X
    )
\]
for any $k$-algebra $C$.
Notice that this description avoids the use of formal schemes, and in
particular it gives a universal arc which is defined as an ordinary scheme. The
category of quasi-compact and quasi-separated schemes is probably large enough
to contain all arc spaces of geometric interest, but the theory developed in
\cite{Bha16} in very delicate, and we preferred to avoid relying on it. Our
results are local in nature, so it is enough for us to have an analogue of the
above formula in the affine case. This is precisely the content of
\cref{th:affine-univ-prop}, which is an elementary fact from commutative
algebra. Our results do not require the quasi-compact and quasi-separated
conditions.
\end{remark}

\section{The sheaves $\mathcal P_n$}

\label{s:sheaf-b}

The goal of this section is to define the sheaves $\mathcal P_n$ appearing in
\cref{th:differentials-arc-space,th:differentials-jet-scheme}. We start by
looking at the affine case. We continue with the notation introduced
in~\cref{s:arcs-jets}, so that given an $R$-algebra $A$ we have $A_n =
\HS^n_{A/R}$, $B_n = A_n[t]/(t^{n+1})$ when $n$ is finite, and $B_\infty =
A_\infty\llbracket t\rrbracket$.  

\begin{definition}
For any $n \in \mathbb N \cup \{\infty\}$, we define $P_n$ to be the $B_n$-module given by
\[
    P_n := t^{-n} A_n[t] / t A_n[t]
\]
when $n$ is finite and 
\[
    P_\infty := A_\infty(\!(t)\!) / t A_\infty\llbracket t\rrbracket
\]
when $n = \infty$.
\end{definition}

As $A_n$-modules, we have $B_n = \prod_{i=0}^n A_n t^i$ and $P_n =
\bigoplus_{j=0}^n A_n t^{-j}$. It is convenient to view an element $b \in B_n$
as a power series $b = \sum_{i=0}^n a_i t^i$ (a polynomial if $n$ is finite),
and an element $p \in P_n$ as a polynomial $p = \sum_{j=0}^n a'_{-j} t^{-j}$.
With this in mind, we can view the $B_n$-module structure as follows: the
action of an element $b \in B_n$ on an element $p \in P_n$ is simply given by
the product $b\cdot p$ of the two series, modulo $t A_n\llbracket t
\rrbracket$. 

Note that the $A_n$-module $\Hom_{A_n}(P_n, A_n)$ has a natural $B_n$-module
structure given by precomposition. That is, given $b \in B_n$ and $\phi \colon
P_n \to A_n$, we define $b\cdot\phi$ to be the homomorphism $P_n \to A_n$
defined by $(b\cdot\phi)(p) := \phi(b\cdot p)$. 

\begin{lemma}
\label{th:B_n-dual-of-P_n-local}
For every $n \in \mathbb N \cup \{\infty\}$, 
there exists a canonical isomorphism $B_n \simeq \Hom_{A_n}(P_n, A_n)$
as $B_n$-modules.
\end{lemma}

\begin{proof}
Since $B_n = \prod_{i=0}^n A_n t^i$ and $P_n = \bigoplus_{j=0}^n A_n t^{-j}$,
there is a canonical isomorphism of $A_n$-modules $B_n \simeq \Hom_{A_n}(P_n,
A_n)$ given by
\[
	b = \sum_{i=0}^n a_i t^i 
	\mapsto
	\left(\phi_b \colon p = \sum_{j=0}^n a'_{-j} t^{-j} 
    \mapsto
    \sum_{i=0}^n a_i a'_{-i}\right),
\]
and it is immediate to check that this isomorphism is compatible with the
respective $B_n$-module structures. 
\end{proof}

\begin{remark}
\label{th:B_n-dual-of-P_n-local-module-M}
\Cref{th:B_n-dual-of-P_n-local} generalizes to all $A_n$-modules, in the
following way. For every $A_n$-module $M$, the space $\Hom_{A_n}(P_n, M)$ has a
natural $B_n$-module structure given by precomposition, and there is a
canonical isomorphism 
\[
M\hat\otimes_{A_n}B_n \simeq \Hom_{A_n}(P_n, M)
\]
as $B_n$-modules. The proof follows the same arguments of the proof of
\cref{th:B_n-dual-of-P_n-local}, once one observes that $M\hat\otimes_{A_n}B_n
= \prod_{i=0}^n M t^i$.
\end{remark}

\begin{remark}
When $n$ is finite, we can view $\{t^{-j}\}_{j=0}^n$ as the dual basis of
$\{t^i\}_{i=0}^n$, and we have $P_n \simeq \Hom_{A_n}(B_n, A_n)$. Note, though, that
$P_\infty$ is not the $A_\infty$-dual of $B_\infty$.
\end{remark}

\begin{lemma}
\label{th:P_n=B_n}
For $n \in \mathbb N$, the morphism that sends $t^{-j}$ to $t^{-j+n}$ gives an
isomorphism of $B_n$-modules between $P_n$ and $B_n$. By contrast, $P_\infty$
and $B_\infty$ are not isomorphic, not even as $A_\infty$-modules. 
\end{lemma}

\begin{proof}
Multiplication by $t^n$ clearly gives an isomorphism of $A_n$-modules $P_n
\simeq B_n$, and one can check that this is compatible with the $B_n$-module
structures. The last assertion is also clear since $P_\infty \simeq
A_\infty[t]$ (as $A_\infty$-module) whereas $B_\infty = A_\infty\llbracket t
\rrbracket$.
\end{proof}

\begin{remark}
\label{r:mu-sharp}
For $m > n$, the homomorphism $\mu^\sharp_{m,n} \colon B_m \longrightarrow B_n \otimes_{A_n} A_m$
defining the morphism $\mu_{m,n} \colon U_n \times_{(X/S)_n} (X/S)_m \to U_m$
corresponds, via the duality given in \cref{th:B_n-dual-of-P_n-local}, to
the inclusion 
\[
    P_n \otimes_{A_n} A_m \longrightarrow P_m
\]
that sends $t^{-j}$ in $P_n \otimes_{A_n} A_m$ to $t^{-j}$ in $P_m$. 
When $m$ is finite, 
this inclusion corresponds via the natural isomorphisms 
$P_n \simeq B_n$ and $P_m \simeq B_m$ to the homomorphism $B_n \otimes_{A_n} A_m \to B_m$
given by multiplication by $t^{m-n}$.
\end{remark}

The definition of $P_n$ globalizes as follows. Given an arbitrary morphism of
schemes $X \to S$, source and target can be covered by affine charts $\Spec A
\subset X$ and $\Spec R \subset S$ so that the morphism is determined by gluing
affine morphisms $\Spec A \to \Spec R$. For every $n$, let $U_n$ be the
universal family given in \cref{eq:univ-jet}. Then the sheaves $P_n$
constructed above for the corresponding charts $\Spec B_n \subset U_n$ glue
together to give a sheaf $\mathcal P_n$ on $U_n$.

For every $n \in \mathbb N \cup \{\infty\}$, there is a natural isomorphism 
\[
	\rho_{n*}(\mathcal O_{U_n}) \simeq
	\Hom_{\mathcal O_{(X/S)_n}}\big(\rho_{n*}(\mathcal P_n),\mathcal O_{(X/S)_n}\big).
\]
Moreover, the right hand side has a natural $\mathcal O_{U_n}$-module structure
given by precomposing with the $\mathcal O_{U_n}$-module action on $\mathcal
P_n$, and with this structure is isomorphic to $\mathcal O_{U_n}$. Furthermore,
if $n$ is finite then we have $\mathcal P_n \simeq \mathcal O_{U_n}$. All these
statements can be checked locally on $X$, and therefore it suffices to consider
the case where $X = \Spec A$ and $S = \Spec R$, where they reduce to
\cref{th:B_n-dual-of-P_n-local,th:P_n=B_n}.

The analysis done in the affine case can be carried out in an identical way in
this more general setting. In particular, for each $m > n$ we get a natural
injective morphism
\[
    \pi_{m,n}^*(\rho_{n*}(\mathcal P_n))
    \longrightarrow 
    \rho_{m*}(\mathcal P_m).
\]
If $m$ and $n$ are finite, then $\mathcal P_m \simeq \mathcal O_{U_m}$ and
$\mathcal P_n \simeq \mathcal O_{U_n}$, and the above injection is conjugate to
the morphism $\pi_{m,n}^*(\rho_{n*}(\mathcal O_{U_n})) \to \rho_{m*}(\mathcal O_{U_m})$
given by multiplication by $t^{m-n}$.

\section{Derivations and differentials}

\label{s:der+diff}

In this section we prove the description of $\Omega_{(X/S)_n}$ stated in
\cref{th:differentials-arc-space,th:differentials-jet-scheme}. We continue with
the notations introduced in the previous section, so that given an $R$-algebra
$A$ we have $A_n = \HS^n_{A/R}$, $B_n = A_n[t]/(t^{n+1})$ when $n$ is finite,
and $B_\infty = A_\infty\llbracket t\rrbracket$.  
As before, we regard $B_n$ an $A_n$-module
via $\rho_n^\sharp$ and as an $A$-module via $\gamma_n^\sharp$,
where these maps are defined in \cref{eq:univ-jet-rings}.

\begin{lemma}
\label{th:derivations}
Let $m,n \in \mathbb N \cup \{\infty\}$.
Let $R$ be a ring and $A$ an $R$-algebra. Let $M$ be an $A_n$-module, and
consider $M \hat\otimes_{A_n} B_n$ with the $A$-module structure induced from
the $A$-module structure on $B_n$ (notice that $\hat\otimes_{A_n} =
\otimes_{A_n}$ when $n$ is finite). Then there is a natural isomorphism
\[
    \Der_R(A_n, M)
    \simeq
    \Der_R(A, M \hat\otimes_{A_n} B_n).
\]
If $m > n$ and $M$ is an $A_m$-module, then the natural map $\Der_R(A_m, M) \to
\Der_R(A_n, M)$ corresponds via the above isomorphism to the map induced by
$\mu_{m,n}^\sharp \colon B_m \to B_n \otimes_{A_n} A_m$.
\end{lemma}

\begin{proof}
To treat the cases of arcs and jets at the same time, we will identify
$R\llbracket t\rrbracket$ with $R[t]/(t^{n+1})$ when $n=\infty$. 

Fix an $A_n$-module $M$ as in the statement of the \lcnamecref{th:derivations},
and consider the $A_n$-module $A_n \oplus \varepsilon M$ with the $A_n$-algebra
structure defined by $(r\oplus \varepsilon m)\cdot(r'\oplus \varepsilon m') =
(rr'\oplus \varepsilon (rm'+r'm))$. The symbol $\varepsilon$ should be thought
as a variable with $\varepsilon^2 = 0$. Since $A_n \hat\otimes_R R[t]/(t^{n+1})
= B_n$, \cref{th:affine-univ-prop} gives a natural isomorphism
\[
    \Hom_{\text{$R$-alg}}(A_n, A_n \oplus \varepsilon M)
    \simeq
    \Hom_{\text{$R$-alg}}(A, B_n \oplus \varepsilon (M \hat\otimes_{A_n}B_n)).
\]
The two modules of derivations that we are interested in are mapped into each
other via this isomorphism. More precisely, we have
\begin{align*}
    \Der_R(A_n, M)
    &\simeq \left\{
        \phi \in \Hom_{\text{$R$-alg}}(A_n, A_n \oplus \varepsilon M)
        \,\,\big|\,\,
        \phi = {\rm id}_{A_n} \bmod \varepsilon
        \right\}
    \\
    &\simeq \left\{
        \phi \in \Hom_{\text{$R$-alg}}(A, B_n \oplus \varepsilon (M \hat\otimes_{A_n}B_n))
        \,\,\big|\,\,
        \phi = \gamma_n^\sharp \bmod \varepsilon
        \right\}
    \\
    &\simeq \Der_R(A, M \hat\otimes_{A_n} B_n).
\end{align*}

For the second statement of the \lcnamecref{th:derivations}, it suffices to note that the
map $\Der_R(A_m, M) \to \Der_R(A_n, M)$ corresponds, via the above
isomorphisms, to the map
\[
	\Hom_{\text{$R$-alg}}(A, B_m \oplus \varepsilon (M \hat\otimes_{A_m}B_m))
	\to 
	\Hom_{\text{$R$-alg}}(A, B_n \oplus \varepsilon (M \hat\otimes_{A_n}B_n))
\]
induced by the projection $R[t]/(t^{m+1}) \to R[t]/(t^{n+1})$, and the latter
is exactly the projection that induces $\mu_{m,n}^\sharp$.

All the isomorphisms in the proof are functorial with respect to all the data
involved, and therefore the resulting isomorphisms are natural. 
\end{proof}

\begin{remark}
The previous \lcnamecref{th:derivations} is the algebraic incarnation of an
intuitive geometric fact about tangent vectors on arc spaces and jet schemes.
For concreteness, we look at the case of arcs when $R = k$ is a field. 
Consider a point $\alpha$ in
$X_\infty$ with residue field $L$, here regarded as an $A_\infty$-module. Then an
element of $\Der_k(A_\infty, L)$ corresponds to a tangent vector to $X_\infty$
at $\alpha$. Using the isomorphism of \cref{th:derivations}, this tangent
vector gets identified with an element of $\Der_k(A, L\llbracket t\rrbracket)$,
which corresponds to a vector field on $X$ along the image of the arc $\alpha$.
These types of identifications are expected for moduli spaces of maps. For
example, given two smooth projective varieties $X$ and $Y$, we can consider the
space $\mathcal M = {\rm Mor}(Y, X)$ parametrizing morphisms from $Y$ to $X$.
Then, for a morphism $f \colon Y \to X$, we have the well-known
formula
\[
    T_{\mathcal M, f} \simeq H^0(Y, f^*T_X),
\]
which is analogous to \cref{th:derivations}.
\end{remark}

\begin{theorem}
\label{th:differentials}
Let $X \to S$ be a morphism of schemes. 
For every $n \in \mathbb N \cup \{\infty\}$ we have a natural isomorphism
\[
    \Omega_{(X/S)_n/S} 
    \simeq 
    \rho_{n*}(\gamma_n^*(\Omega_{X/S}) 
    \otimes
    \mathcal P_n)
\]
where $\rho_n \colon U_n \to (X/S)_n$ and $\gamma_n \colon U_n \to X$ are
defined in \cref{eq:univ-jet}, and these sheaves are isomorphic to
$\rho_{n*}(\gamma_n^*(\Omega_{X/S}))$ whenever $n$ is finite. Moreover, for
every $m,n \in \mathbb N \cup \{\infty\}$ with $m
> n$ the morphisms 
\[
\pi_{m,n}^*(\Omega_{(X/S)_n/S}) \to \Omega_{(X/S)_m/S}
\]
induced by the truncation maps are obtained from the natural inclusion
$\pi_{m,n}^*(\rho_{n*}(\mathcal P_n)) \to \rho_{m*}(\mathcal P_m)$ by tensoring
with $\rho_{n*}(\gamma_n^*(\Omega_{X/S}))$, and correspond to the maps
$\pi_{m,n}^*(\rho_{n*}(\mathcal O_{U_n})) \to \rho_{m*}(\mathcal O_{U_m})$
given by multiplication by $t^{m-n}$ whenever $m$ is finite.
\end{theorem}

\begin{proof}
Since these properties are local in $X$, we can assume that $X = \Spec A$ and
$S = \Spec R$. Recall that all $B_n$-modules are regarded as $A_n$-modules
via $\rho_n^\sharp$ and as $A$-modules via $\gamma_n^\sharp$. 
With the same notation as in \cref{s:sheaf-b}, let $M$ be an
arbitrary $A_n$-module. By \cref{th:B_n-dual-of-P_n-local-module-M}, the
natural morphism
\[
	M \hat\otimes_{A_n} B_n \longrightarrow \Hom_{A_n}(P_n, M)
\]
is an isomorphism of $B_n$-modules. Then, by \cref{th:derivations}, we have
a chain of natural isomorphisms
\begin{align*}
    \Hom_{A_n} (\Omega_{A_n/R}, M)
    &\simeq
    \Der_R(A_n, M)
    \\&\simeq
    \Der_R(A, M \hat\otimes_{A_n} B_n)
    \\&\simeq
    \Hom_{A} (\Omega_{A/R}, M \hat\otimes_{A_n} B_n)
    \\&\simeq
    \Hom_{A} (\Omega_{A/R}, \Hom_{A_n}(P_n, M))
    \\&\simeq
    \Hom_{A_n} (\Omega_{A/R} \otimes_A P_n, M).
\end{align*}
It follows that there is a natural isomorphism of $A_n$-modules
$\Omega_{A_n/R} \simeq \Omega_{A/R} \otimes_A P_n$. 
Since $\Omega_{A/R} \otimes_A P_n = (\Omega_{A/R} \otimes_A B_n) \otimes_{B_n} P_n$,
this and the fact that, by \cref{th:P_n=B_n}, $P_n$ is isomorphic to $B_n$ as a
$B_n$-module if $n$ is finite, give the first statement. The other statements
follow immediately from the second part of \cref{th:derivations,r:mu-sharp}.
\end{proof}

\begin{remark}
\label{r:computing-fibers}
Suppose for simplicity that $X = \Spec A$ is affine over $S = \Spec R$, 
so that the formula in \cref{th:differentials} becomes, for $n = \infty$, 
\[
\Omega_{A_\infty/R} \simeq \Omega_{A/R} \otimes_A P_\infty.
\]
For every $A_\infty$-algebra $L$ (e.g., a field corresponding
to a point of $(X/S)_\infty$), we have
\begin{align*}
\Omega_{A_\infty/R} \otimes_{A_\infty} L 
&\simeq \Omega_{A/R} \otimes_A P_\infty \otimes_{A_\infty} L \\
&\simeq \Omega_{A/R} \otimes_A L\llbracket t\rrbracket  \otimes_{L\llbracket t\rrbracket } (P_\infty \otimes_{A_\infty} L).
\end{align*}
Note that $L\llbracket t\rrbracket  = B_\infty \,\hat\otimes_{A_\infty} L$. 
In the second step above we have used that 
$P_\infty \otimes_{A_\infty} L = L(\!(t)\!)/tL\llbracket t\rrbracket $
and hence it is not just a module over $B_\infty \otimes_{A_\infty} L$
but also over $L\llbracket t\rrbracket $.
The above formula will be useful in order to compute fibers of $\Omega_{X_\infty/S}$.
\end{remark}

It is worthwhile to work out explicitly what \cref{th:differentials} is telling us
in the affine case, when both $X$ and $S$ are affine and $X$ is of finite type over $S$. 
We start with the case when $n$ is finite.
If $S = \Spec R$ and $X = \Spec A$, then we can write
\[
A = \frac{R[x_1,\dots,x_r]}{(f_1,\dots,f_s)}.
\]
Then 
\[
A_n = \frac{R[x_i^{(p)} \mid i=1,\dots,r, \; p = 0,\dots, n]}
{(f_j^{(q)} \mid j=1,\dots,s, \; q = 0, \dots, n )}
\]
where for every $g$ and $k$ we set $g^{(k)} := D_k(g)$.
The presentation of $A$ yields the following presentation for $\Omega_{A/R}$:
\[
\bigoplus_j A\,df_j \xrightarrow{\;J\;} \bigoplus_i A\,dx_i \to \Omega_{A/R} \to 0.
\]
Here $J = [\partial f_j/\partial x_i]$ is the Jacobian matrix.
Similarly, the presentation of $A_n$ gives
\[
\bigoplus_{j,q} A_n\,df_j^{(q)} \xrightarrow{\;\widetilde J_n\;} 
\bigoplus_{i,p} A_n\,dx_i^{(p)} \to \Omega_{A_n/R} \to 0
\]
where $\widetilde J_n = [\partial f_j^{(q)}/\partial x_i^{(p)}]$.
\cref{th:differentials} provides an efficient
way of writing down  this matrix. 
Explicitly, the \lcnamecref{th:differentials}
gives the presentation
\[
\bigoplus_j P_n\,df_j \xrightarrow{\;J(t)\;} \bigoplus_i P_n\,dx_i \to \Omega_{A_n/R} \to 0
\]
where $J(t)$ is, entry by entry, the pull-back of $J$
via the universal jet. In particular, we can write
\[
J(t) = J + J't + J''t^2 + \dots + J^{(n)}t^n
\]
where $J^{(k)} = [D_k(\partial f_j/\partial x_i)]$. 
Using the decomposition $P_n = \bigoplus_{k=0}^n A_n t^{-k}$, we get
\[
\bigoplus_{j,q} A_n\,t^{-q}\,df_j \xrightarrow{\;J_n\;} 
\bigoplus_{i,p} A_n\,t^{-p}\,dx_i \to \Omega_{A_n/R} \to 0
\]
where 
\[
J_n = 
\begin{bmatrix}
J & J' & \cdots & J^{(n)} \\
0 & J & \cdots & J^{(n-1)} \\
\vdots & \vdots & \ddots & \vdots \\
0 & 0 &  \cdots & J
\end{bmatrix}.
\]
The isomorphism in \cref{th:derivations} maps
$\partial/\partial x_i^{(p)} \mapsto t^p\,\partial/\partial x_i$. 
By duality, we see that the isomorphism in \cref{th:differentials} 
maps $dx_i^{(p)} \mapsto t^{-p}\,dx_i$
and $df_j^{(q)} \mapsto t^{-q}\,df_j$. From this, we see that
\[
\widetilde J_n = J_n.
\]
In particular, one immediately recovers from this that
\[
\frac{\partial f_j^{(q)}}{\partial x_i^{(p)}} 
= \frac{\partial f_j^{(q+1)}}{\partial x_i^{(p+1)}}
\quad \text{for all} \quad p,q < n
\]
and
\[
\frac{\partial f_j^{(q)}}{\partial x_i^{(p)}} 
= D_{q-p}\left(\frac{\partial f_j}{\partial x_i}\right)
\quad \text{for all} \quad p<q.
\]
These relations are well-known and 
can be viewed as an elementary consequence of the chain rule.
It is a fun exercise is to verify them 
starting from an explicit polynomial. 
The above presentation of $\Omega_{A_n/R}$ is compatible
with the differentials $\Omega_{A_n/R} \otimes_{A_n} A_m \to \Omega_{A_m/R}$
of the truncation maps and can be used to compute them. 
Letting $n \to \infty$, one gets a similar description of the presentation
matrix for $\Omega_{A_\infty/R}$ where now the matrix is infinite dimensional.

\section{Invariant factors and Fitting invariants}

\label{s:fitting}

In the next section we will be interested in studying fibers of the sheaves of
differentials on jet schemes $(X/S)_n$. Using \cref{th:differentials}, this
will involve understanding the pull-back of $\Omega_{X/S}$ along a jet.
As a preparation, in this section we include some remarks on these types of
pull-backs.

Let $X$ be an arbitrary scheme over base scheme $S$. For a given 
$n \in \mathbb N \cup \{\infty\}$, 
consider a point $\alpha_n \in (X/S)_n$, and let $L_n$
denote the residue field of $\alpha_n$. We do not assume that $\alpha_n$ is a
closed point of $(X/S)_n$. By shrinking $X$ around $\psi_n(\alpha_n)$, we may
assume without loss of generality that $X = \Spec A$ and $S = \Spec R$ are
affine. Let $A_n$ and $B_n$ be the algebras defined in \cref{s:sheaf-b}, so
that $(X/S)_n = \Spec A_n$ and $U_n = \Spec B_n$. 

Algebraically, $\alpha_n$ is given by a morphism
\[
    \alpha_n^\sharp
    \colon
    A
    \longrightarrow
    B_n \,\hat\otimes_{A_n} L_n
\]
where $B_n \,\hat\otimes_{A_n} L_n = L_n[t]/(t^{n+1})$ if $n$ is finite and
$B_\infty \,\hat\otimes_{A_\infty} L_\infty = L_\infty\llbracket t \rrbracket$.
For simplicity, in what follows we will identify $L_\infty\llbracket t
\rrbracket$ with $L_n[t]/(t^{n+1})$ when $n = \infty$, and always use the more
suggestive notation $L_n[t]/(t^{n+1})$ instead of $B_n \hat\otimes_{A_n} L_n$.

Consider a finitely generated $A$-module $M$. We are interested in
understanding the structure of its pull-back along $\alpha_n$, which is given
by
\[
    M \otimes_A L_n[t]/(t^{n+1}).
\]
Notice that $L_n[t]/(t^{n+1})$ is a principal ideal ring (and a domain when $n
=\infty$). Since $M$ is finitely generated, the pull-back is also finitely
generated, and the structure theory for finitely generated modules over
principal ideal rings gives a unique decomposition
\[
    M \otimes_A L_n[t]/(t^{n+1})
    \;\simeq\;
    \big(L_n[t]/(t^{n+1})\big)^d
    \;\oplus\;
    \bigoplus_{i \geq d} L_n[t]/(t^{e_i})
\]
where the direct sum in the right hand side has finitely many nonzero terms and
$n+1 > e_d \geq e_{d+1} \geq \cdots \geq e_{d+r}$. When $0 \leq i < d$ we set $e_i =
n+1$ (so $e_i = \infty$ if $n=\infty$). If the dependency on $\alpha_n$ and $M$
needs to be emphasized, we will write $d = d(\alpha_n) = d(\alpha_n, M)$ and
$e_i = e_i(\alpha_n) = e_i(\alpha_n, M)$.

\begin{definition}
We call $\{e_i(\alpha_n, M)\}_{i=0}^\infty$ the sequence of \emph{invariant
factors} of $M$ with respect to $\alpha_n$. The number $d(\alpha_n, M)$ is
called the \emph{Betti number} of $M$ with respect to $\alpha_n$.
\end{definition}

The invariant factors of $M$ with respect to $\alpha_n$ determine the pull-back
$M \otimes_A L[t]/(t^{n+1})$ up to isomorphism.

If $\alpha_n$ is the truncation of another jet $\alpha_m$ (so
$\pi_{m,n}(\alpha_m) = \alpha_n$), the invariant factors with respect to
$\alpha_n$ and $\alpha_m$ are related. We have:
\[
    e_i(\alpha_n)
    =
    \min\{
        n+1,\,
        e_i(\alpha_m)
    \}.
\]
Notice that the Betti numbers with respect to $\alpha_n$ and $\alpha_m$ could
be different. We always have $d(\alpha_n) \geq d(\alpha_m)$.

The invariant factors are related to the Fitting ideals of $M$, whose
definition we recall briefly. Since $M$ is finitely generated, we can find a
presentation
\begin{equation}
\label{eq:presentation}
    F_1 
    \overset{\varphi}{\longrightarrow}
    F_0 \longrightarrow M \to 0,
\end{equation}
where $F_0$ and $F_1$ are free $A$-modules and $F_0$ is finitely generated.
Then $\Fitt^i(M) \subset A$ is the ideal generated by the minors of size
$\operatorname{rank}(F_0)-i$ of the matrix representing $\varphi$.
Geometrically, a point $x \in X$ belongs to the zero locus of $\Fitt^i(M)$ if
and only if the dimension of the fiber $M \otimes_A k(x)$ is $> i$.

Recall that if $J \subset A$ is an ideal, $c = \ord_{\alpha_n}(J)$ is defined
as the number for which $\alpha_n^\sharp(J) = (t^c)$. When $n$ is finite we
pick $c \in \{0,1,\ldots,n+1\}$, and when $n=\infty$ we have $c \in
\{0,1,\ldots,\infty\}$ (we use the convention $t^\infty = 0$).

\begin{definition}
Consider the numbers
\[
    c_i(\alpha_n, M) := \ord_{\alpha_n}(\Fitt^i(M)).
\]
We call $\{c_i(\alpha_n, M)\}_{i=0}^\infty$ the sequence of \emph{Fitting
invariants} of $M$ with respect to $\alpha_n$.
\end{definition}

The Fitting invariants are determined by the invariant factors. To see this,
consider the pull-back via $\alpha_n$ of the presentation in
\cref{eq:presentation}:
\begin{equation}
\label{eq:presentation2}
    \widetilde F_1 
    \overset{\widetilde\varphi}{\longrightarrow}
    \widetilde F_0 \longrightarrow
    M \otimes_{A} L_n[t]/(t^{n+1})
    \to 0.
\end{equation}
The structure theory says that, after picking appropriate bases,
$\widetilde\varphi$ is represented by a matrix with entries $t^{e_0}, t^{e_1},
t^{e_2}, \ldots$ along the main diagonal and zeros elsewhere. The minors of
$\widetilde\varphi$ are the pull-backs of the minors of $\varphi$. This
property is expressed by saying that ``the formation of Fitting ideals commutes
with base change'', and it gives that
\[
    \Fitt^i(M \otimes_R L_n[t]/(t^{n+1}))
    =
    \alpha_n^\sharp(\Fitt^i(M)) \cdot L_n[t]/(t^{n+1})
    =
    (t^{c_i}) \;\subset\; L_n[t]/(t^{n+1}),
\]
where $c_i = \ord_{\alpha_n}(\Fitt^i(M))$. Here the Fitting ideals of $M
\otimes_A L_n[t]/(t^{n+1})$ are computed with respect to its structure as a
module over $L_{n}[t]/(t^{n+1})$. On the other hand, we can compute these
Fitting ideals directly using the presentation in \cref{eq:presentation2}. We
get:
\[
    c_i = \min\{n+1,\, e_i + e_{i+1} + e_{i+2} + \cdots \}.
\]

If $n = \infty$, we get that $c_i = e_i + e_{i+1} + e_{i+2} + \cdots$, and we
see that in this case the invariant factors are determined by the Fitting
invariants. The Betti number counts the number of infinite Fitting invariants
and can be interpreted geometrically as the dimension of the fiber $M \otimes_A
k(\xi)$, where $\xi = \alpha_\infty(\eta)$ is the generic point of the arc
$\alpha_\infty$.
In particular, we have that $\ord_\alpha(\Fitt^d(M)) < \infty$.

\section{The fiber over a jet}

\label{s:fiber}

In this section, we assume that $X$ is a scheme of finite type over an
arbitrary base scheme $S$. This condition on $X$ guarantees that the sheaf of
differentials $\Omega_{X/S}$ is a finitely generated $\mathcal O_X$-module. In
particular, we can consider its Fitting ideals and can apply the results of
\cref{s:fitting}.

\begin{remark}
To compute $\Fitt^i(\Omega_{X/S})$, one can work locally on $X$ and $S$, and
use a relatively closed embedding of $X$ in some affine space $\mathbb A^N_S$
over $S$ to get a presentation as in \cref{eq:presentation} where $\varphi$ is
the Jacobian matrix of the embedding. As for any module, a point $x \in X$
belongs to the zero locus of $\Fitt^i(\Omega_{X/S})$ if and only if the
dimension of the fiber $\Omega_{X/S} \otimes_{\mathcal O_X} k(x)$ is $> i$. In
particular, if $X$ is a reduced and equidimensional scheme of finite type over
a field $k$, then $\Fitt^i(\Omega_{X/k})$ is zero when $i < \dim X$, and equals
the Jacobian ideal $\Jac_X$ when $i = \dim X$.
\end{remark}

We are interested in studying the fibers of $\Omega_{(X/S)_n/S}$, and relating
them to the Fitting invariants of $\Omega_{X/S}$.

\begin{theorem}
\label{th:fitting}
Let $X$ be a scheme of finite type over a base scheme $S$, let $n \in \mathbb N$, 
and consider a jet
$\alpha_n
\in (X/S)_n$. We do not assume that $\alpha_n$ is a closed point of
$(X/S)_n$, and we let $L_n$ be its residue field. Let $d_n$ and $\{e_i\}$ be the
Betti number and invariant factors of $\Omega_{X/S}$ with respect to $\alpha_n$.
Then
the isomorphism $\Omega_{(X/S)_n/S} \simeq \rho_{n*}(\gamma_n^*(\Omega_{X/S}))$
given by \cref{th:differentials} induces an isomorphism
\[
    \Omega_{(X/S)_n/S} \otimes_{\mathcal O_{(X/S)_n}} L_n
    \;\simeq\;
    \big(L_n[t]/(t^{n+1})\big)^{d_n}
    \;\oplus\;
    \bigoplus_{i \geq d_n} L_n[t]/(t^{e_i}).
\]
\end{theorem}

\begin{proof}
After restricting to a suitable open set of $X$, we can assume that $X = \Spec
A$ and $S = \Spec R$. As before, we use the notation from \cref{s:sheaf-b}.
From \cref{th:differentials}, since $n$ is finite we see that $\Omega_{A_n/S}
\simeq \Omega_{A/S} \otimes_A B_n$, where $B_n$ is considered as an $A$-module
via $\gamma_n^\sharp$. This implies that 
\[
	\Omega_{A_n/S} \otimes_{A_n} L_n
	\simeq \Omega_{A/S} \otimes_A L_n[t]/(t^{n+1}),
\] 
where $L_n[t]/(t^{n+1})$ is considered as an $A$-module via $\alpha_n^\sharp$.
The \lcnamecref{th:fitting} now
follows from the definition of the invariant factors with respect to the jet
$\alpha_n$.
\end{proof}

We now restrict ourselves to
\emph{liftable jets}. By definition, these are points in a jet scheme $(X/S)_n$
that lie in the image of the truncation map $\pi_n \colon (X/S)_\infty \to
(X/S)_n$.
From the above \lcnamecref{th:fitting} it is immediate to compute the dimensions of
the fibers of $\Omega_{(X/S)_n/S}$ over liftable jets. 

\begin{corollary}
\label{th:fiber-dim-gen}
In addition to the assumptions of \cref{th:fitting}, assume that $\alpha_n =
\pi_n(\alpha)$ for some arc $\alpha \in (X/S)_\infty$.
Consider the ideal sheaf
$\mathcal J_{d_n} := \Fitt^{d_n}(\Omega_{X/S})$. Then\[
    \dim_{L_n}\big( \Omega_{(X/S)_n/S} \otimes_{\mathcal O_{(X/S)_n}} L_n \big)
    =
    (n+1)d_n + \ord_\alpha(\mathcal J_{d_n}).
\]
\end{corollary}

\begin{proof}
Starting at the position $i=d_n$ we have an equality of invariant factors
$e_i(\alpha_n) = e_i(\alpha)$.
Since $\ord_{\alpha}(\mathcal J_{d_n}) = c_{d_n} = e_{d_n} + e_{d_n+1} +
\cdots$,
the result follows immediately from \cref{th:fitting}.
\end{proof}

\begin{remark}
The Betti number $d_n = d(\alpha_n, \Omega_{X/S})$ appearing in the previous
two results is hard to interpret in geometric terms. On the other hand, the
Betti number $d = d(\alpha, \Omega_{X/S})$ has a clear meaning: it is
the dimension of the fiber $\Omega_{X/S} \otimes_{\mathcal O_X} k(\xi)$, where
$\xi = \alpha(\eta) \in X$ is the generic point of $\alpha$. Recall that when
$n$ is large enough (bigger than all the invariant factors of $\alpha$) the
Betti numbers of $\Omega_{X/S}$ with respect to $\alpha$ and $\alpha_n = \pi_n(\alpha)$
coincide.
\end{remark}

The next corollary recovers \cite[Proposition~5.1]{dFD14}.

\begin{corollary}
\label{th:fiber-dim}
In addition to the assumptions of \cref{th:fiber-dim-gen}, assume that $S = \Spec k$
for a field $k$, that $X$ is reduced and equidimensional over $k$, and that the
arc $\alpha$ is not completely contained in the singular locus of $X$. Then,
for finite $n \geq \ord_\alpha(\Jac_X)$ we have
\[
    \dim_{L_n}\big( \Omega_{X_n/k} \otimes_{\mathcal O_{X_n}} L_n \big)
    =
    (n+1)\dim X + \ord_\alpha(\Jac_X).
\]
\end{corollary}

\begin{proof}
With the additional assumptions, we see that the Betti number 
of $\Omega_{X/S}$ with respect to $\alpha$ is $d = \dim X$, and therefore 
$\mathcal J_d = \Fitt^d(\Omega_{A/k}) = \Jac_X$. 
The condition $n \geq \ord_\alpha(\Jac_X)$ guarantees that the Betti numbers of
$\Omega_{A/k}$ with respect to $\alpha$ and $\alpha_n$ coincide.
The result is just a
restatement of \cref{th:fiber-dim-gen} in this case.
\end{proof}

\section{Embedding dimension}

\label{s:embdim}

We now study embedding dimensions of arcs and jets. Starting with this section
and for the reminder of the paper, we assume that $X$ is a scheme of finite
type over a perfect field $k$. 

In the following, let $\alpha \in X_\infty$ be a point, and denote by $L =
L_\infty$ the residue field of $\alpha$. We do not assume that $\alpha$ is a
closed point of $X_\infty$. For every $n \in \mathbb N$, we let $\alpha_n = \pi_n(\alpha)$
be the truncations, and denote by $L_n$ their residue fields. It will be convenient to
also allow the notation $\alpha_\infty$ for $\alpha$. 

For $n \in \mathbb N$, we denote $\dim(\alpha_n) := \trdeg(L_n/k)$. Since the ground
field $k$ is assumed to be perfect, we have
\[
    \dim(\alpha_n) = \dim_{L_n}(\Omega_{L_n/k}).
\]

We start with some preliminary lemmas. For ease of notation, in the discussion
of these preliminary properties we restrict ourselves to the affine setting and
assume that $X = \Spec A$ where $A$ is a finitely generated $k$-algebra. We
apply the notation from \cref{s:sheaf-b} with $R = k$. 

For each $n \in \mathbb N \cup \{\infty\}$, we let $I_n \subset A_n$ be the prime ideal defining $\alpha_n$.
When $m > n$ we have inclusions $I_n \subset I_m$. The Zariski tangent space of
$X_n$ at $\alpha_n$ is the dual of the $L_n$-vector space $I_n/I_n^2$, and
hence the embedding dimension of $X_n$ at $\alpha_n$ is given by
\[
    \embdim( \mathcal O_{X_n, \alpha_n} )
    = 
    \dim_{L_n}(I_n/I_n^2).
\]
Note that there are natural maps $I_n/I_n^2 \to I_m/I_m^2$ whenever $m > n$, and
\[
    I_\infty/I_\infty^2
    = \injlim_{n\to\infty} (I_n/I_n^2).
\]

\begin{lemma}
\label{th:emb-dim-jets-gen}
As above, let $X = \Spec A$ where $A$ is a finitely generated $k$-algebra, and let $\alpha \in X_\infty$. 
For every $n \in \mathbb N$
let $d_n$ be the Betti number of $\Omega_{A/k}$ with
respect to the truncation $\alpha_n = \pi_n(\alpha)$, and consider the ideal
$J_{d_n} :=
\Fitt^{d_n}(\Omega_{A/k})$. Then
\[
    \embdim( \mathcal O_{X_n, \alpha_n} )
    =
    (n+1)d_n - \dim(\alpha_n)
    + \ord_\alpha(J_{d_n}).
\]
\end{lemma}

\begin{proof}
Applying \cite[Theorem~25.2]{Mat89} to the sequence $k \to
(A_n)_{I_n} \to L_n$, we get an exact sequence
\[
    0 
    \to 
    I_n/I_n^2
    \longrightarrow
    \Omega_{A_n/k} \otimes_{A_n} L_n
    \longrightarrow
    \Omega_{L_n/k}
    \to 0.
\]
Here we used the assumption that $k$ is perfect. The
\lcnamecref{th:emb-dim-jets-gen} now follows from \cref{th:fiber-dim-gen} and
the equality $\dim(\alpha_n) = \dim_{L_n}(\Omega_{L_n/k})$.
\end{proof}

\begin{lemma}
\label{th:seq-diff}
With the same assumptions as \cref{th:emb-dim-jets-gen},
let $d$ and $e_i$ be the Betti number and invariant factors of $\Omega_{A/k}$ with
respect to the arc $\alpha$, and let $\Omega_{A_n/k}
\otimes_{A_n} A_m \to \Omega_{A_m/k}$ be the map induced by the truncation
morphism $\pi_{m,n} \colon X_m \to X_n$. Then, for finite
$m \geq n + \ord_\alpha(J_d)$, we have
\[
K := \ker \left(\Omega_{A_n/k} \otimes_{A_n} L \to \Omega_{A_m/k} \otimes_{A_m} L\right)
\simeq 
\left(
\frac{L[t]}{(t^{n+1})}
\right)^{d_n-d}
\oplus
\left(
\bigoplus_{i \geq d_n} \frac{L[t]}{(t^{e_i})}
\right),
\]
and hence $\dim_L(K) = (n+1)(d_n-d) + \ord_\alpha(J_{d_n})$.
In particular, for $n \ge \ord_\alpha(J_d)$ and $m \geq n + \ord_\alpha(J_d)$ we have
\[
K = \ker \left(\Omega_{A_n/k} \otimes_{A_n} L \to \Omega_{A_m/k} \otimes_{A_m} L\right)
\simeq 
\bigoplus_{i \geq d} \frac{L[t]}{(t^{e_i})}
\]
and $\dim_L(K) = \ord_\alpha(J_d)$.
\end{lemma}

\begin{proof}
Since $m \geq \ord_\alpha(J_d)$, we see that the Betti numbers of
$\Omega_{A/k}$ with respect to $\alpha$ and $\alpha_m$ coincide.
By \cref{th:fitting,th:differentials} (see also \cref{r:mu-sharp}), 
we see that $K$ is isomorphic to the kernel
of the map 
\[
\Bigg(\frac{L[t]}{(t^{n+1})}\Bigg)^{\!\!d} 
\oplus
\Bigg(\frac{L[t]}{(t^{n+1})}\Bigg)^{\!\!d_n-d} 
\oplus
\Bigg(\bigoplus_{i \geq d_n} \frac{L[t]}{(t^{e_i})}\Bigg)
\xrightarrow{\cdot t^{m-n}}
\Bigg(\frac{L[t]}{(t^{m+1})}\Bigg)^{\!\!d} 
\oplus
\Bigg(\bigoplus_{i=d}^{d_n-1} \frac{L[t]}{(t^{e_i})}\Bigg)
\oplus
\Bigg(\bigoplus_{i \geq d_n} \frac{L[t]}{(t^{e_i})}\Bigg)
\]
given by multiplication by $t^{m-n}$. Since we have $m-n \ge \ord_\alpha(J_d)
\ge e_i$ for all $i \ge d$, the first assertion follows. Notice that $\sum_{i
\geq d_n} e_i = \ord_\alpha(J_{d_n})$. 
The last assertion follows form this and the fact that if $n \ge \ord_\alpha(J_d)$
then $d_n = d$.
\end{proof}

\begin{lemma}
\label{th:lambda}
With the same assumptions as \cref{th:seq-diff}, consider the natural
morphism
\[
    \lambda_n \colon
    I_n/I_n^2 \otimes_{A_n} L \longrightarrow I_\infty/I_\infty^2
\]
induced by the truncation map. 
Then
\[
    \dim_L(\Im(\lambda_n))
    \ge
    (n+1)d - \dim(\alpha_n).
\]
\end{lemma}

\begin{proof}
Consider $m \geq n + \ord_\alpha(J_d)$. We have the following commutative
diagram with exact rows and columns:
\[\xymatrixrowsep{1.5pc}
    \xymatrix{
        & 0 \ar[d]
        & 0 \ar[d]
        & 0 \ar[d]
        &
        \\
        0 \ar[r]
        & K' \ar[r] \ar[d]
        & K \ar[r] \ar[d]
        & K'' \ar[d]
        &
        \\
        0 \ar[r]
        & I_n/I^2_n \otimes L \ar[r] \ar[d]^{\lambda_{n,m}}
        & \Omega_{A_n/k} \otimes L \ar[r] \ar[d] 
        & \Omega_{L_n/k} \otimes L \ar[r] \ar[d] 
        & 0
        \\
        0 \ar[r]
        & I_{m}/I^2_{m} \otimes L \ar[r] 
        & \Omega_{A_{m}/k} \otimes L \ar[r] 
        & \Omega_{L_{m}/k} \otimes L \ar[r] 
        & 0
        \\
    }
\]
Recall that by \cref{th:seq-diff,th:emb-dim-jets-gen} we have
\begin{align*}
\dim_L(I_n/I^2_n \otimes L)
&= (n+1)d_n - \dim(\alpha_n) + \ord_\alpha(J_{d_n}), \\
\dim_L(K)
&= (n+1)(d_n-d) + \ord_\alpha(J_{d_n}).
\end{align*}
Then from the first column of the diagram we see that
\begin{align*}
	\dim_L(\Im(\lambda_{n,m})) 
	&= \dim_L(I_n/I^2_n \otimes L) - \dim_L(K') \\
	&\ge \dim_L(I_n/I^2_n \otimes L) - \dim_L(K) \\
	&= (n+1)d - \dim(\alpha_n).
\end{align*}
Since $\Im(\lambda_n) = \injlim_m\Im(\lambda_{n,m})$, the result follows.
\end{proof}

\begin{lemma}
\label{th:embdim-limit}
With the same assumptions as \cref{th:seq-diff}, we have 
\[
	\embdim(\mathcal O_{X_\infty, \alpha}) \ge 
        \limsup_{n \to \infty} \big((n+1)d - \dim(\alpha_n)\big).
\]
\end{lemma}

\begin{proof}
Consider the maps $\lambda_n$ of \cref{th:lambda}. Since $I_\infty/I_\infty^2 =
\injlim (I_n/I_n^2)$, we also have that $I_\infty/I_\infty^2 = \injlim
(\Im(\lambda_n))$. Therefore the assertion follows from \cref{th:lambda}.
\end{proof}

We now return to the global case of schemes of finite type over $k$. 
Recall that the dimension of a scheme $X$ of finite type over $k$ at a point $\xi$, denoted
$\dim_\xi(X)$, is defined as the infimum of the dimensions (over $k$) of all open
neighborhoods of $\xi$ in $X$. 

\begin{theorem}
\label{th:embdim>=dimX}
Assume that $X$ is a scheme of finite type over $k$. Then
for any arc $\alpha \in X_\infty$, letting 
$\xi = \alpha(\eta) \in X$ denote the generic point of the arc, we have
\[
    \embdim(\mathcal O_{X_\infty, \alpha}) 
    \geq
    \dim_\xi(X) - \dim(\alpha_0).
\]
\end{theorem}

\begin{proof}
By \cref{th:embdim-limit} and the fact that the Betti number of $\Omega_{X/k}$
with respect to $\alpha$ is bounded below by $\dim_\xi(X)$.
\end{proof}

Since the ground field if perfect, a
point $x$ on a scheme of finite type $X$ is singular if and only if
$\dim(\Omega_X \otimes k(x)) > \dim_x(X)$. In particular, non-reduced points of
$X$ are singular. We denote by $\Sing X$ the singular locus of $X$.
Without further mention, we will use the fact that, as the ground field is perfect, 
if $X$ is a variety then $\Sing X$ is a proper closed subset of $X$ and hence, in particular, has smaller dimension.

\begin{lemma}
\label{th:dim-a_n}
Let $X$ be a a scheme of finite type over $k$, let $\alpha \in X_\infty$, 
and let $Z \subset X$ be the closure of the generic point $\alpha(\eta)$ of $\alpha$. 
Then for every $n \in \mathbb N$ we have
\[
\dim(\alpha_n) \le (n+1)\dim(Z).
\]
\end{lemma}

\begin{proof}
In characteristic zero the statement follows easily from
Kolchin's Irreducibility Theorem \cite{Kol73}.
Over an arbitrary perfect field, the argument can be adjusted as follows. 
Note that $\alpha \in Z_\infty \setminus (\Sing Z)_\infty$. 
Then, by \cite[Theorem~2.9]{Reg09}, $\alpha$ belongs to the unique irreducible 
component $C$ of $Z_\infty$ that dominates $Z$.
It follows that the closure of the projection of $C$ to $Z_n$
has dimension $(n+1)\dim(Z)$, and the assertion follows.
\end{proof}

\begin{proposition}
\label{th:embdim=infty}
Let $X$ be scheme of finite type,
let $\alpha \in X_\infty$, and denote by $\xi = \alpha(\eta) \in X$ the generic point of $\alpha$. 
Assume that one of the following two conditions holds:
\begin{enumerate}
\item
$\alpha \in Y_\infty$ where $Y \subset X$ is a closed subscheme 
with $\dim_\xi(Y) < \dim_\xi(X)$; or
\item
$\alpha \in (\Sing X)_\infty$.
\end{enumerate}
Then 
\[
    \embdim(\mathcal O_{X_\infty, \alpha}) = \infty.
\]
\end{proposition}

\begin{proof}
Suppose first that~(1) holds. By the geometric interpretation of
Betti numbers, we have $d(\alpha,\Omega_X) \ge \dim_\xi(X)$. 
Since $\dim(\alpha_n) \le (n+1)\dim_\xi(Y)$ by \cref{th:dim-a_n}, and 
$\dim_\xi(Y) < \dim_\xi(X)$, \cref{th:embdim-limit} implies that
\[
	\embdim(\mathcal O_{X_\infty, \alpha}) \ge
        \limsup_{n\to\infty}(n+1)(d(\alpha,\Omega_X)-\dim(Y)) = \infty.
\]

Suppose then that~(2) holds. If
$\dim_\xi(\Sing X)_\infty < \dim_\xi (X)$, then the assertion follows from case~(1).
Otherwise, $X$ is non-reduced at $\xi$ and hence
$d(\alpha,\Omega_X) > \dim _\xi (X)$, 
and we conclude that $\embdim(\mathcal O_{X_\infty,\alpha}) =
\infty$ by \cref{th:embdim-limit,th:dim-a_n}.
\end{proof}

\begin{remark}
\label{r:equalities}
The inequalities stated in \cref{th:lambda,th:embdim-limit} are both
equalities whenever the map $K' \to K$
in the proof of \cref{th:lambda} is an isomorphism. This clearly the case
if $L_m$ is separable over $L_n$, as in this case $K'' = 0$, and therefore we get equalities in both formulas
if $k$ has characteristic zero or $\alpha$ is a $k$-rational point of $X_\infty$. 
More interestingly, we will see later in \cref{s:stable-points}
that the condition that $L_m$ is separable over $L_n$ is always guaranteed
if $X$ is a variety and $\alpha$ is a stable point of $X_\infty$, and hence we 
will deduce that equalities hold in this case, too. 
It will turn out in the end (see \cref{th:embdim-limit:equality}) that in fact
the equality holds in general, and the limsup is a limit, in the formula stated in \cref{th:embdim-limit}.
\end{remark}

\section{The birational transformation rule}

\label{s:BTR}
Here we study how birational morphisms affect the embedding dimension of arcs.

Given two schemes $X$ and $Y$
of finite type over a perfect field $k$, we say that a morphism
$f \colon Y \to X$ is \emph{birational over a union of components} if there
exist a dense open set $V \subset Y$ and a (not necessarily dense) open set $U
\subset X$ such that $f(V) \subset U$ and the restriction $f|_V \colon V \to U$
is an isomorphism. If $U$ is dense in $X$, or equivalently if $f$ is dominant,
then we say that $f$ is \emph{birational}. 

If $f \colon Y \to X$ is birational over a union of components, then the sheaf
of relative differentials $\Omega_{Y/X}$ is torsion. We consider the \emph{Jacobian ideal} 
$\Jac_f := \Fitt^0(\Omega_{Y/X})$ of $f$. 

\begin{lemma}
\label{th:same-residue}
Let $X$ and $Y$ be schemes of finite type over a perfect field, and
consider a proper map $f \colon Y \to X$ that is birational over a union of
components. Let $\beta \in Y_\infty$ and consider $\alpha = f_\infty(\beta) \in
X_\infty$. If $f$ is locally an isomorphism at the generic point $\beta(\eta)$ of the arc
(that is, one can pick $V \subset Y$ as above such that $\beta(\eta) \in V$),
then the residue fields of $\alpha$ and $\beta$ are equal.
\end{lemma}

\begin{proof}
Let $L$ and $K$ be the residue fields of $\alpha$ and $\beta$, respectively.
Since $\alpha = f_\infty(\beta)$, we have $L \subset K$. Consider $\alpha$ as
map $\alpha \colon \Spec L\llbracket t\rrbracket \to X$. The hypothesis on
$\beta$ guarantees that the generic point $\alpha(\eta)$ of $\alpha$ lies in
the locus over which $f$ is an isomorphism, and therefore it can be lifted. The
valuative criterion of properness gives a unique lift of $\alpha$ to an arc
$\widetilde\alpha \colon \Spec L\llbracket t \rrbracket \to Y$. This
corresponds to a morphism $\Spec L \to Y_\infty$ whose image is $\beta$ by
construction. This implies that $K \subset L$, as required.
\end{proof}

\begin{theorem}
\label{th:BTR}
Let $X$ and $Y$ be schemes of finite type over a perfect field.
Consider a proper map $f \colon Y \to X$ that is birational over a union of
components. Let $\beta \in Y_\infty$ and consider $\alpha = f_\infty(\beta) \in
X_\infty$. Assume that $Y$ is smooth at $\beta(0)$. Then
\[
    \embdim\left(\mathcal O_{X_\infty,\alpha}\right)
    =
    \embdim\left(\mathcal O_{Y_\infty,\beta}\right)
    +
    \ord_{\beta}(\Jac_f).
\]
\end{theorem}

Notice that several of these numbers could be infinite. For example, if $\beta$
has infinite embedding dimension, the \lcnamecref{th:BTR} implies that $\alpha$
also has infinite embedding dimension. Conversely, if $\alpha$ has infinite
embedding dimension and $\beta$ is not completely contained in the 
vanishing locus of $\Jac_f$, then $\beta$ has infinite embedding dimension.

\begin{proof}[Proof of \cref{th:BTR}]
Let $V \subset Y$ be a dense open subset as above, so that $f$ restricts
to an isomorphism from $V$ to an open set $U \subset X$.
If $\beta$ is contained in $Z := Y \setminus V$, then $\alpha$ is contained in
the image of $Z$, and hence both embedding dimensions are
infinite by \cref{th:embdim=infty}. Thus we assume that $\beta$
is not contained in $Z$, which means that $\beta(\eta) \in V$. By \cref{th:same-residue} both
$\alpha$ and $\beta$ have the same residue field, which we call $L$. Let $I$
and $J$ be the ideals defining $\alpha$ and $\beta$. Since the ground field is
perfect, we have the following diagram:
\begin{equation}
    \label{eq:BTR-main}
    \begin{gathered}
    \xymatrix{
        0 \ar[r]
        & I/I^2 \ar[r] \ar[d]
        & \Omega_{X_\infty/k} \otimes_{\mathcal O_{X_\infty}} L \ar[r] \ar[d]^{\varphi}
        & \Omega_{L/k} \ar[r] \ar@{=}[d]
        & 0 \\
        0 \ar[r]
        & J/J^2 \ar[r]
        & \Omega_{Y_\infty/k} \otimes_{\mathcal O_{Y_\infty}} L \ar[r]
        & \Omega_{L/k} \ar[r]
        & 0 \\
    }
    \end{gathered}
\end{equation}
Therefore, the \lcnamecref{th:BTR} will follow if we show that
\[
    \dim_L(\ker \varphi) - \dim_L(\coker \varphi) = \ord_\beta(\Jac_f).
\]

Recall the universal arc $\rho_\infty \colon U_\infty \to X_\infty$ and the
sheaf $\mathcal P_\infty$ on $U_\infty$ defined in \cref{s:sheaf-b}. 
We rely on \cref{r:computing-fibers} for the computation 
of the fibers appearing in the middle column of \cref{eq:BTR-main}.
For ease of notation, we denote
\[
    B_L 
    := L\llbracket t \rrbracket 
    = \rho_{\infty*}(\mathcal O_{U_\infty}) \,\hat\otimes_{\mathcal O_{X_\infty}} L
\]
and
\[
    P_L :
    = L(\!(t)\!)/t L\llbracket t\rrbracket
    = \rho_{\infty*}(\mathcal P_\infty) \otimes_{\mathcal O_{X_\infty}} L.
\]
We can regard $B_L$ both as an $\mathcal O_X$-algebra via the arc $\alpha$ and
as an $\mathcal O_Y$-algebra via $\beta$. Then $P_L$, which is naturally a
$B_L$-module, becomes both an $\mathcal O_X$-module and an $\mathcal
O_Y$-module.

The map $f$ induces a natural sequence of sheaves of differentials:
\[
    \xymatrix{
    \Omega_{X/k} \otimes_{\mathcal O_X} \mathcal O_Y
    \ar[r]^-{f^*}
    & \Omega_{Y/k}
    \ar[r]
    & \Omega_{Y/X}
    \ar[r]
    & 0.
    }
\]
After pulling back to the arcs $\alpha$ and $\beta$ we get:
\[
    \xymatrix{
    \Omega_{X/k} \otimes_{\mathcal O_X} B_L
    \ar[r]^-{\psi}
    & \Omega_{Y/k} \otimes_{\mathcal O_Y} B_L
    \ar[r]
    & \Omega_{Y/X} \otimes_{\mathcal O_Y} B_L
    \ar[r]
    & 0.
    }
\]
All of the terms in this sequence are finitely generated modules over $B_L =
L\llbracket t \rrbracket$, and therefore they are direct sums of cyclic
modules. 
Since $Y$ is smooth at $\beta(0)$, the middle term $F_Y := \Omega_{Y/k}
\otimes_{\mathcal O_Y} B_L$ is free. Write $\Omega_{X/k} \otimes_{\mathcal O_X}
B_L = F_X \oplus T_X$, where $F_X$ is free and $T_X$ is torsion. 
Since $f$ is an isomorphism at the generic point of $\beta$,
the restriction $\overline\psi = \psi|_{F_X}$ is injective. Consider $Q_{Y/X} =
\coker(\overline\psi)$.
We have an exact sequence:
\begin{equation}
    \label{eq:BTR-Q}
    \xymatrix{
    0 \ar[r]
    & F_X
    \ar[r]^-{\overline\psi}
    & F_Y
    \ar[r]
    & Q_{Y/X}
    \ar[r]
    & 0.
    }
\end{equation}
Notice that
$\psi(T_X)
= 0$, and therefore $Q_{Y/X} = \Omega_{Y/X}\otimes_{\mathcal O_Y} B_L$.

\Cref{th:differentials} says that $\varphi$ is obtained from $\psi$ by
tensoring with $P_L$:
\begin{equation}
    \label{eq:BTR-phi}
    \begin{gathered}
    \xymatrix@C-=0.6cm{
    0
    \ar[r]
    & K 
    \ar[r]
    & \Omega_{X/k} \otimes_{\mathcal O_X} P_L
    \ar[r]^-{\varphi} \ar@{=}[d]
    & \Omega_{Y/k} \otimes_{\mathcal O_Y} P_L
    \ar[r] \ar@{=}[d]
    & \Omega_{Y/X} \otimes_{\mathcal O_Y} P_L
    \ar[r]
    & 0
    \\
    && \Omega_{X_\infty/k} \otimes_{\mathcal O_{X_\infty}} L
    & \Omega_{Y_\infty/k} \otimes_{\mathcal O_{Y_\infty}} L
    }
    \end{gathered}
\end{equation}

Notice that $P_L$ is a divisible $B_L$-module, and hence tensoring with $P_L$
kills torsion. We get the following diagram:
\begin{equation}
    \label{eq:BTR-phi-bar}
    \begin{gathered}
    \xymatrix@C-=0.6cm{
    0
    \ar[r]
    & K 
    \ar[r] \ar@{=}[d]
    & F_X \otimes_{B_L} P_L
    \ar[r]^-{\overline\varphi} \ar@{=}[d]
    & F_Y \otimes_{B_L} P_L
    \ar[r] \ar@{=}[d]
    & Q_{Y/X} \otimes_{B_L} P_L
    \ar[r] \ar@{=}[d]
    & 0
    \\
    0
    \ar[r]
    & K 
    \ar[r]
    & \Omega_{X/k} \otimes_{\mathcal O_X} P_L
    \ar[r]^-{\varphi}
    & \Omega_{Y/k} \otimes_{\mathcal O_Y} P_L
    \ar[r]
    & \Omega_{Y/X} \otimes_{\mathcal O_Y} P_L
    \ar[r]
    & 0,
    }
    \end{gathered}
\end{equation}
where $\overline\varphi$ is induced by $\overline\psi$.

Since $Q_{Y/X} = \Omega_{Y/X} \otimes_{\mathcal O_Y} B_L$ is torsion (because
$\beta$ is not contained in the exceptional locus), we see that $Q_{Y/X}
\otimes_{B_L} P_L = 0$. Moreover, $K = \operatorname{Tor}_1^{B_L}(Q_{Y/X}, P_L)
= Q_{Y/X}$. Therefore $\dim_L(K) = \dim_L(Q_{Y/X}) = \ord_\beta(\Jac_f)$, and
the result follows.

Alternatively, we can check directly that $\dim_L(K) = \ord_\beta(\Jac_f)$.
To do this, notice that, after appropriate choices of bases, the map
$\overline\psi$ can be represented by a matrix with entries $t^{e_0}, t^{e_1},
\ldots$ along the main diagonal an zeroes elsewhere. In the language of
\cref{s:fitting}, the $e_i$ can be chosen to be the invariant factors of the
module $\Omega_{Y/X}$ with respect to the arc $\beta$. In particular, we have
that $\ord_\beta(\Jac_f) = \sum_{i \geq 0} e_i$. Since $\beta$ is not contained
in the exceptional locus, we have $e_i < \infty$ for all $i$.
The map $\overline\varphi$ is represented by the same matrix as
$\overline\psi$. We get that $K = \bigoplus_{i \geq 0} K_i$, where $K_i$ is the
kernel of the map $P_L \to P_L$ given by multiplication by $t^{e_i}$. An easy
computation shows that $\dim_L(K_i) = e_i$, and therefore $\dim_L(K) = \sum_{i \geq
0} e_i = \ord_\beta(\Jac_f)$.
\end{proof}

\begin{theorem}
\label{th:BTR-gen}
Let $X$ and $Y$ be schemes of finite type over a perfect field.
Consider a proper map $f \colon Y \to X$ that is birational over a union of
components. Let $\beta \in Y_\infty$ and consider $\alpha = f_\infty(\beta) \in
X_\infty$. Then
\[
    \embdim\left(\mathcal O_{Y_\infty,\beta}\right)
    \leq
    \embdim\left(\mathcal O_{X_\infty,\alpha}\right)
    \leq
    \embdim\left(\mathcal O_{Y_\infty,\beta}\right)
    + \ord_{\beta}(\Jac_f).
\]
\end{theorem}

\begin{proof}
The proof is almost identical to the one of \cref{th:BTR}. 
Using the notation
of that \lcnamecref{th:BTR}, the main difference is that $\Omega_{Y/k}
\otimes_{\mathcal O_Y} B_L$ is no longer a free $B_L$-module.
Write
$\Omega_{Y/k}
\otimes_{\mathcal O_Y} B_L = F_Y \oplus T_Y$, where $F_Y$ is a free
$B_L$-module and $T_Y$ is torsion, and let $\overline\psi$ be the composition
of $\psi|_{F_X}$ with the projection to $F_Y$. Then $\overline\psi$ is still
injective, and we can consider the module $Q_{Y/X}$ given by the sequence in
\cref{eq:BTR-Q}. The diagrams in \cref{eq:BTR-phi,eq:BTR-phi-bar} remain valid.

Notice that $\Omega_{Y/X} \otimes_{\mathcal O_Y} B_L$ is a torsion
$B_L$-module, so $\coker{\varphi} = 0$, and the diagram of \cref{eq:BTR-main}
shows that $ \embdim\left(\mathcal O_{Y_\infty,\beta}\right) \leq
\embdim\left(\mathcal O_{X_\infty,\alpha}\right)$.

The module $Q_{Y/X}$ is a quotient of $\Omega_{Y/X} \otimes_{\mathcal O_Y}
B_L$, and therefore it is also torsion. As in the proof of \cref{th:BTR}, this
implies that:
\[
    \ker(\varphi) 
    = K 
    = \operatorname{Tor}_1^{B_L}(Q_{Y/X}, P_L) 
    = Q_{Y/X}.
\]
Using again that $Q_{Y/X}$ is a quotient of 
$\Omega_{Y/X} \otimes_{\mathcal O_Y} B_L$, we see that
\[
    \dim_L(K) 
    = \dim_L(Q_{Y/X}) 
    \leq
    \dim_L(\Omega_{Y/X} \otimes_{\mathcal O_Y} B_L)
    = \ord_\beta(\Jac_f),
\]
and the result follows.
\end{proof}

\begin{corollary}
\label{th:bijection-finite-embdim}
Let $f \colon Y \to X$ be a proper birational morphism between schemes
of finite type over a perfect field. Then the induced map $f_\infty \colon
Y_\infty \to X_\infty$ induces a bijection
\[
	\{ \beta \in Y_\infty \mid \embdim\left(\mathcal O_{Y_\infty,\beta}\right) < \infty \}
	\xrightarrow{1-1}
	\{ \alpha \in X_\infty \mid \embdim\left(\mathcal O_{X_\infty,\alpha}\right) < \infty \}.
\]
\end{corollary}

\begin{proof}
\Cref{th:BTR-gen} implies that 
$Y_\infty$ has finite embedding dimension at a point $\beta$
if and only if $X_\infty$ has finite embedding dimension at $f_\infty(\beta)$.
To conclude, it suffices to observe that if $\alpha \in X_\infty$ is
not in the image of $f_\infty$ then $X_\infty$ has infinite embedding dimension at $\alpha$.
Indeed, by the valuative criterion of properness,
$\alpha$ must be fully contained in the indeterminacy locus of $f^{-1} \colon X
\dashrightarrow Y$. Since $f$ is a birational map, the
indeterminacy locus of $f^{-1}$ has dimension strictly smaller than the
dimension of $X$ at any of its points, and therefore we have $\embdim(\mathcal
O_{X_\infty,\alpha}) = \infty$ by \cref{th:embdim=infty}.
\end{proof}

\Cref{th:BTR-gen} also implies the following useful property. 

\begin{corollary}
\label{th:BTR-inclusion}
Let $X$ be a scheme of finite type over a perfect field $k$ and $\alpha \in X_\infty$ an 
arc such that $\alpha \not\in (\Sing X)_\infty$. 
Let $Y \subset X$ be the irreducible
component of $X$ such that $\alpha \in Y_\infty$.
Then
\[
    \embdim\left(\mathcal O_{Y_\infty,\alpha}\right)
    =
    \embdim\left(\mathcal O_{X_\infty,\alpha}\right).
\]
\end{corollary}

\begin{proof}
The first remark, which is implicit in the statement, 
is that for every $\alpha \in X_\infty$ there is 
always an irreducible component $Y$ of $X$ such that $\alpha \in Y_\infty$, and
if $\alpha \not\in (\Sing X)_\infty$ then such component is unique.

If $f \colon Y \to X$ is the natural inclusion, then $\Jac_f = \mathcal O_Y$ and hence the
assertion follows directly from \cref{th:BTR-gen}.
\end{proof}

\cref{th:BTR-inclusion} would follow immediately if we knew that the local rings
$\mathcal O_{X_\infty,\alpha}$ and $\mathcal O_{Y_\infty,\alpha}$ are isomorphic. 
This is however not true in general. 
While the reductions of these rings can be shown to be 
isomorphic, it is not always the case that the rings
themselves are. An example where this fails is given next.

\begin{example}
Let $X = \{xy=0\} \subset \mathbb A^2$ be the union of two lines 
and $Y = \{y=0\} \subset \mathbb A^2$ one of the components, 
and consider the arc $\alpha = (-t,0) \in X_\infty$.
Let $\Gamma = \mathbb Z \oplus \epsilon \mathbb Z$ with the lexicographic order, 
so that, for instance, 
\[
\dots < 0 < \epsilon < 2\epsilon < \dots < 1 - 2\epsilon < 1 - \epsilon < 1 < 1 + \epsilon < \dots,
\]
and let $R \subset k(r,s)$ be the rank 2 valuation ring with value group $\Gamma$
associated to the monomial valuation $v$ defined by $v(r) = \epsilon$
and $v(s) = 1$. 
The map $k[x,y]/(xy) \to (R/(rs))\llbracket t\rrbracket $ defined by 
\[
x \mapsto r - t, \quad
y \mapsto s + \frac s r\, t + \frac s{r^2} \,t^2 + \frac s{r^3}\, t^3 + \dots
\]
does not factor through the quotient $k[x,y]/(xy) \to k[x]$, and
this shows that the corresponding map $\mathcal O_{X_\infty,\alpha} \to R/(rs)$
does not factor through the quotient $\mathcal O_{X_\infty,\alpha}
\to \mathcal O_{Y_\infty,\alpha}$.
\end{example}

\begin{remark}
If $Y$ is a union of components
of a non-reduced scheme $X$, the inclusion $f \colon Y \to X$ is not
necessarily a birational map over a union of components. 
In particular, \cref{th:BTR,th:BTR-gen} may not apply. 
Consider for instance the case where $X$ is nowhere reduced
and $Y = X_\red$. If $f$ is the inclusion then $\Jac_f = \mathcal O_Y$
and hence $\ord_{\alpha}(\Jac_f) = 0$ for every $\alpha \in Y_\infty$. 
As we will see in the next section, we can always find 
an arc $\alpha \in Y_\infty$ with $\embdim\left(\mathcal O_{Y_\infty,\alpha}\right) < \infty$. 
However, since $\Sing X = Y$, for any $\alpha$ we have 
$\embdim\left(\mathcal O_{X_\infty,\alpha}\right) = \infty$ by \cref{th:embdim=infty}.
\end{remark}

\section{Constructible and stable points}

\label{s:stable-points}

Throughout this section, let $k$ be a perfect field and 
$X$ a scheme of finite type over $k$. The goal of this section is  
to characterize local rings $\mathcal O_{X_\infty,\alpha}$ of finite embedding dimension. 
We start by recalling the notion of contructibility. 

Let $Z$ be an arbitrary scheme. The definition of constructible subset of 
$Z$ is here intended in the sense of \cite{EGAiii_1} 
(see also \cite[\href{https://stacks.math.columbia.edu/tag/005G}{Tag~005G}]{SP}).
That is, a subset of $Z$ is \emph{constructible} if and only if it is a finite union of
finite intersections of retrocompact open sets and their complements, where a
subset $W \subset Z$ is said to be \emph{retrocompact} if for every
quasi-compact open set $U \subset Z$, the intersection $W \cap U$ is quasi-compact. 
A constructible set is \emph{irreducible} if it contains a unique \emph{generic point}, 
namely, a point whose closure in $Z$ contains the set. 
With small abuse of terminology, we say that a point $z \in Z$ is 
\emph{constructible} if $z$ is the
generic point of an irreducible constructible subset of $Z$.
Note that the fact that a point $z \in Z$ is constructible in this sense
does not mean necessarily that the 
1-point subset $\{z\}$, or its closure, are constructible subsets of $Z$. 

%
%

The above definition gives a notion of what it means for a point $\alpha \in X_\infty$
to be a constructible point. 
It is a general fact that a subset of $X_\infty$ is
constructible if and only if it is the (reduced) inverse image of a
constructible subset of $X_n$ for some finite $n$
\cite[Th\'eor\`eme~(8.3.11)]{EGAiv_3}. 

\begin{lemma}
\label{t:stable-point-irred-comp}
Let $X$ be a scheme of finite type and $\alpha \in X_\infty$ an 
arc such that $\alpha \not\in (\Sing X)_\infty$. 
Let $Y \subset X$ be the irreducible
component of $X$ such that $\alpha \in Y_\infty$.
Then
$\alpha$ is a constructible point of $X_\infty$ if and only if it 
is a constructible point of $Y_\infty$.
\end{lemma}

\begin{proof}
Suppose first that $\alpha$ is a constructible point of $X_\infty$. 
Then $\alpha$ is the generic point of an irreducible constructible subset $C \subset X_\infty$.
Note that $C \cap Y_\infty$ is a constructible subset of $Y_\infty$, 
and since $\alpha \in Y_\infty$ and $Y_\infty$ is closed in $X_\infty$, 
we actually have $C \subset Y_\infty$.
It follows that $\alpha$ is a constructible point of $Y_\infty$.

Conversely, suppose that $\alpha$ is a constructible point of $Y_\infty$, 
and let $D \subset Y_\infty$ be an irreducible constructible subset
with generic point $\alpha$.
We can find an integer $n$ such that $D = (\pi_n|_{Y_\infty})^{-1}(D_n)$
where $D_n$ is a constructible subset of $Y_n$. 
Since $X_n$ is a scheme of finite type and $Y_n$ is a closed subscheme of $X_n$, 
the set $D_n$ is also constructible in $X_n$. 
Since $\alpha \not\in (\Sing X)_\infty$, 
we can find a closed subscheme 
$Z \subset X$ such that $X = Y \cup Z$ and $\alpha \not\in Z_\infty$.
This means that $\alpha$ has finite order of contact with $Z$. 
If $m$ is an integer larger than this order, then 
$\alpha \not\in \pi_m^{-1}(Z_m)$.
Then $\pi_n^{-1}(D_n) \setminus \pi_m^{-1}(Z_m)$
is a constructible subset of $X_\infty$
with generic point $\alpha$, and
hence $\alpha$ is a constructible point of $X_\infty$. 
\end{proof}

When $X$ is a variety, constructible subsets of $X_\infty$ 
are called \emph{weakly stable semi-algebraic sets} in \cite{DL99} and
\emph{cylinders} in \cite{ELM04,Ish08,dFEI08,EM09} among other places.
Related notions are those of \emph{generically stable set} introduced in \cite{Reg06}
and \emph{quasi-cylinder} introduced in \cite{dFEI08}.
The more restrictive notion of \emph{stable semi-algebraic set}
was introduced in \cite{DL99}. We recall this notion next. 

Following the terminology of \cite{DL99} (see also \cite{EM09}), 
a morphism $g \colon V' \to V$ of schemes of finite type is said to 
induce a \emph{piecewise trivial fibration} $W' \to W$ with fiber $F$,
where $W' \subset V'$ and $W \subset V$ are constructible subsets and $F$
is a reduced scheme, if there is a decomposition $W = T_1 \sqcup \dots \sqcup T_r$, 
with all $T_i$ locally closed subsets of $W$ such that each 
$W' \cap g^{-1}(T_i)$ is locally closed in $V'$ and, with the reduced scheme structure, 
it is isomorphic to $T_i \times F$. 

Let $X$ be a variety. A subset $W \subset X_\infty$ is said to be a \emph{stable semi-algebraic subset}
if for all $n \gg 1$ the truncation $\pi_{n+1}(X_\infty) \to \pi_n(X_\infty)$
induce a piecewise trivial fibration over $\pi_n(W)$ with fiber $\mathbb A^{\dim(X)}$. 
The term \emph{stable point} was coined in \cite{Reg09} to denote the generic point of
an irreducible stable semi-algebraic subset of $X_\infty$.
As explained in \cite[(2.7)]{DL99}, a simple consequence of \cite[Lemma~4.1]{DL99}
is that a point $\alpha \in X_\infty$ is stable if and only if it is constructible 
and is not contained in $(\Sing X)_\infty$.
Note that \cite[Lemma~4.1]{DL99} is stated in characteristic zero, but the proof
works over an arbitrary perfect field (cf.\
\cite[Proposition~4.1]{EM09} where the property is proved over algebraically closed fields
of arbitrary characteristic).

\begin{remark}
The formulation of \cite[3.1(i)(c)]{Reg09} is equivalent
to asking that the point is the generic point of a generically stable subset of $X_\infty$
as defined in \cite{Reg06}. It seems that 
the condition that the generic point of 
a generically stable subset be not contained in $(\Sing X)_\infty$
was overlooked in \cite{Reg06,Reg09}. 
An easy and well-known computation of the fiber over the 1-jet $(t,0,0)$ 
in the arc space of the Whitney umbrella $X = \{xy^2 = z^2\} \subset \mathbb A^3$
shows that such condition is necessary.
\end{remark}

\begin{remark}
\label{th:stable-separable-extensions}
Let $X$ be a variety and $\alpha \in X_\infty$ a stable point.
For every $n \in \mathbb N$, let $\alpha_n = \pi_n(\alpha)$ be the truncation of $\alpha$
and $L_n$ its residue field. Then it follows immediately from the definition
that for all $n \gg 1$ the field extension 
$L_{n+1} \subset L_n$ is purely transcendental of degree $\dim(X)$
(cf.\ \cite[\S 3.1(i)(b)]{Reg09}). 
\end{remark}

\begin{remark}
The reason for introducing the term \emph{constructible point} in the context of arc spaces is that it
is not completely clear what the definition of stable point of $X_\infty$ should be if
$X$ is a nonreduced scheme. The point is that the difference between weakly stable and stable
when $X$ is a variety can be characterized in two equivalent ways: either by 
imposing the condition that $\alpha$ is not contained in $(\Sing X)_\infty$, 
or by requiring the piecewise trivial fibration condition with fiber $\mathbb A^{\dim(X)}$. 
If $X$ is not generically reduced, these two conditions are no longer equivalent. 
\end{remark}

Implicit in the works on motivic integration is the definition of codimension 
of a constructible subset of the arc space of a smooth variety. 
This was formalized and extended to singular varieties in 
\cite{ELM04,dFEI08,dFM15}.
Let $X$ be a variety and $\alpha \in X_\infty$, and for any $n \in \mathbb N$ let
$\alpha_n = \pi_n(\alpha)$. 
The \emph{jet codimension} of $\alpha$ in $X_\infty$ is defined to be
\[
    \jetcodim(\alpha, X_\infty) := \lim_{n \to \infty} \big((n+1)\dim(X) - \dim(\alpha_n)\big).
\]
The fact that the limit exists is an easy application of \cite[Lemma~4.1]{DL99}
(e.g., see \cite[Lemma~4.13]{dFM15}). 
Note also that $\jetcodim(\alpha, X_\infty) \ge 0$ for every $\alpha \in X_\infty$
(cf.\ \cref{th:dim-a_n}).

The next property, which is well-known to experts,
formalizes the relationship between stable points and jet codimension.

\begin{proposition}
\label{th:stable-points}
Let $X$ be a variety and $\alpha \in X_\infty$. 
Then $\alpha$ is a stable point if and only if $\jetcodim(\alpha, X_\infty) < \infty$.
\end{proposition}

\begin{proof}
If $\alpha \not\in(\Sing X)_\infty$ then this follows easily from \cite[Lemma~4.1]{DL99}. 
If $\alpha \in(\Sing X)_\infty$, then, since the ground 
field is perfect, the closure $Z \subset X$ of the generic point
$\alpha(\eta)$ of $\alpha$ has dimension $\dim(Z) < \dim(X)$ and
the same argument applied to $Z$ implies that $\jetcodim(\alpha, X_\infty) = \infty$.
\end{proof}

The following more precise version of \cref{th:embdim-limit} holds on varieties.

\begin{lemma}
\label{th:embdim-limit:equality}
Let $X$ be a variety, $\alpha \in X_\infty$ any point, and
$d$ the Betti number of $\Omega_{X/k}$ with respect to $\alpha$. 
Then
\[
	\embdim(\mathcal O_{X_\infty, \alpha}) =
        \lim_{n \to \infty} \big((n+1)d - \dim(\alpha_n)\big).
\]
\end{lemma}

\begin{proof}
Since the right hand side of the equation is bounded below by 
the jet codimension, it follows by \cref{th:embdim-limit,th:dim-a_n,th:stable-points}
that both terms of the equation are infinite unless $\alpha$ is a stable point.
If $\alpha$ is a stable point, then \cref{th:stable-separable-extensions} implies that 
for $m \ge n \gg 1$ the field extensions $L_n \subset L_m$ are separable, and therefore
equality holds by \cref{r:equalities}.
\end{proof}

\begin{theorem}
\label{th:embdim=cylcodim}
Let $X$ be a variety over a perfect field $k$. Then for every
$\alpha \in X_\infty$ we have
\[
    \embdim(\mathcal O_{X_\infty, \alpha}) =
    \jetcodim(\alpha, X_\infty).
\]
\end{theorem}

\begin{proof}
If $\alpha \in (\Sing X)_\infty$, then both sides of the equation are
infinite by \cref{th:embdim=infty,th:stable-points}. 
Assume then that $\alpha \not\in (\Sing X)_\infty$. 
By \cref{th:embdim-limit} and the fact that $d(\alpha,\Omega_{X/k}) = \dim(X)$, we have
\begin{align*}
	\embdim(\mathcal O_{X_\infty, \alpha}) 
	&\ge \lim_{n \to \infty} \big((n+1)\dim(X) - \dim(\alpha_n)\big) \\
	&= \jetcodim(\alpha, X_\infty).
\end{align*}
If $\alpha$ is stable then the inequality in the first step of this formula
is an equality by \cref{th:embdim-limit:equality}, and hence, combining the formula with
\cref{th:stable-points}, we get
\[
    \embdim(\mathcal O_{X_\infty, \alpha}) =
    \jetcodim(\alpha, X_\infty) < \infty.
\]
If $\alpha$ is not stable, then we have $\jetcodim(\alpha, X_\infty) = \infty$ by 
\cref{th:stable-points}, and we conclude that
$\embdim(\mathcal O_{X_\infty, \alpha}) = \jetcodim(\alpha, X_\infty) = \infty$.
\end{proof}

We obtain the following characterization of local rings of finite embedding
dimension.

\begin{theorem}
\label{th:char-finite-cylcodim}
Let $X$ be a scheme of finite type over a perfect field. 
For every $\alpha \in X_\infty$, we have 
\[
	\embdim(\mathcal O_{X_\infty,\alpha}) < \infty
\] 
if and only if $\alpha$ is a constructible point and is not contained in $(\Sing
X)_\infty$.
\end{theorem}

\begin{proof}
If $\alpha \in (\Sing X)_\infty$ then $\embdim(\mathcal O_{X_\infty,\alpha}) = \infty$
by \cref{th:embdim=infty}. 

Assume then that $\alpha \not\in (\Sing X)_\infty$. This implies that $X$ is
reduced and irreducible at the generic point $\xi = \alpha(\eta)$. By
\cref{th:BTR-inclusion,t:stable-point-irred-comp}, 
we can replace $X$ with its irreducible component
containing $\xi$ and thus assume that it is a variety. Then \cref{th:embdim=cylcodim} gives us 
$\embdim(\mathcal O_{X_\infty,\alpha}) =
\jetcodim(\alpha)$, and we conclude from the fact that $\jetcodim(\alpha) <
\infty$ if and only if $\alpha$ is stable by \cref{th:stable-points}.
\end{proof}

\begin{remark}
\label{th:min-primes-inf-embdim}
By \cite[Theorem~2.9]{Reg09}, if $X$ is a variety defined over a
perfect field of positive characteristic, then $X_\infty$ has finitely many
irreducible components only one of which is not contained in $(\Sing
X)_\infty$. An example where $X_\infty$ has more than one component 
is given by the $p$-fold Whitney umbrella $X = \{xy^p = z^p\} \subset \mathbb A^3$ in 
characteristic $p$, see \cite[Example~8.1]{dF}. \Cref{th:char-finite-cylcodim} implies that if
$\alpha$ is the generic point of an irreducible component of $X_\infty$ that is contained
in $(\Sing X)_\infty$, then $\mathcal O_{X_\infty,\alpha}$, 
which is a zero dimensional ring, has infinite embedding dimension.
\end{remark}

\begin{corollary}
\label{th:embdim=embdim+Jac_X}
Let $X$ be a variety over a perfect field. For any $\alpha \in X_\infty$ 
and $n \in \mathbb N$, let $\alpha_n = \pi_n(\alpha)$ be the truncation of $\alpha$. 
If $\embdim(\mathcal O_{X_\infty, \alpha}) = \infty$, then 
$\embdim(\mathcal O_{X_n, \alpha_n})$ becomes arbitrarily large as $n$ increases. 
Otherwise, we have
\[
    \embdim(\mathcal O_{X_\infty, \alpha}) 
    = \embdim(\mathcal O_{X_n, \alpha_n}) - \ord_\alpha(\Jac_X)
\]
for all sufficiently large integers $n$. 
\end{corollary}

\begin{proof}
For short, let $d_n := d(\alpha_n,\Omega_{X/k})$ be the Betti number.
If $\embdim(\mathcal O_{X_\infty, \alpha}) = \infty$, then 
$\jetcodim(\alpha,X_\infty) = \infty$ by \cref{th:embdim=cylcodim}, and hence, 
since $d_n \ge \dim(X)$ for all $n$ and the sequence of numbers $\ord_\alpha(J_{d_n})$
stabilizes for $n$ large enough,
$\embdim(\mathcal O_{X_n, \alpha_n})$ goes to $\infty$ as $n \to \infty$ by \cref{th:emb-dim-jets-gen}. 
Otherwise, $\alpha$ is a stable point and $d_n = \dim(X)$ for 
$n \gg 1$ by \cref{th:char-finite-cylcodim}, and
hence the corollary follows 
by \cref{th:emb-dim-jets-gen,th:embdim-limit:equality}.
\end{proof}

\begin{corollary}
Let $X$ be a variety over a perfect field and
$f \colon Y \to X$ a resolution of singularities. 
Then for every $\beta \in Y_\infty$, 
letting $\alpha = f_\infty(\beta)$, we have
\[
    \jetcodim(\alpha, X_\infty)
    =
    \jetcodim(\beta, Y_\infty)
    +
    \ord_{\beta}(\Jac_f).
\]
\end{corollary}

\begin{proof}
By \cref{th:embdim=cylcodim,th:BTR}.
\end{proof}

One of the nice features of local rings of finite embedding dimension comes from
the following elementary fact.

\begin{lemma}
\label{th:noeth=fin-emb-dim}
For any scheme $Z$ over a field, the completion $\widehat{\mathcal O_{Z,z}}$ of
the local ring of $Z$ at a point $z$ is Noetherian if and only if
$\embdim(\mathcal O_{Z,z}) < \infty$.
\end{lemma}

\begin{proof}
This is in fact a general result about completions of local rings. Let
$(\widehat R,\widehat{\mathfrak m})$ be the $\mathfrak m$-adic completion of a
local ring $(R,\mathfrak m)$. If $\mathfrak m/\mathfrak m^2$ is finite
dimensional, then $\widehat{\mathfrak m}$ is finitely generated by
\cite[\href{http://stacks.math.columbia.edu/tag/0315}{Tag 0315}]{SP}, and this
implies that $\widehat R$ is Noetherian. The converse follows by the fact that
since $\widehat{\mathfrak m}^2 \subset \widehat{\mathfrak m^2}$, there is
always a surjection $\widehat{\mathfrak m}/\widehat{\mathfrak m}^2 \to
\mathfrak m/\mathfrak m^2$.
\end{proof}

The following property is an immediate consequence of \cref{th:char-finite-cylcodim}. 

\begin{corollary}
\label{th:Noetherian}
Let $X$ be a reduced scheme of finite type over a perfect field. The completion 
$\widehat{\mathcal O_{X_\infty,\alpha}}$
of the local ring at a point $\alpha \in X_\infty$ is Noetherian if and only if
$\alpha$ is a constructible point and is not contained in $(\Sing X)_\infty$.
\end{corollary}

\begin{proof}
From \cref{th:char-finite-cylcodim,th:noeth=fin-emb-dim}.
\end{proof}

When $X$ is a variety, the fact that that the completion of the local ring at a stable point
$\alpha \in X_\infty$ is Noetherian
is a result of Reguera. It follows from \cite[Corollary~4.6]{Reg06},
which proves 
that $\widehat{\mathcal O_{(X_\infty)_\red,\alpha}}$
is Noetherian (cf.\ \cite[\S2.3, (vii)]{MR}), and \cite[Theorem~3.13]{Reg09},
which proves that there is an isomorphism
$\widehat{\mathcal O_{X_\infty,\alpha}} \simeq \widehat{\mathcal O_{(X_\infty)_\red,\alpha}}$.
Notice that this last result of Reguera is stated in characteristic zero, 
but the proof extends to all perfect fields using Hasse--Schmidt derivations. 

The Curve Selection Lemma 
\cite[Corollary~4.8]{Reg06} easily follows from
Cohen's Structure Theorem once one knows that 
these rings
are Noetherian. It is a powerful statement that allows the study of certain
containments among sets in the arc space via the use of arcs in the arc space.
All the current proofs solving the Nash
problem in dimension $2$ \cite{FdBPP12,dFD16} use the Curve Selection Lemma in
an essential way.

We conclude with the 
following property which was obtained by different methods in \cite[Proposition~4.1]{Reg09}. 
A more general statement in characteristic zero which applies to maps that are 
not necessarily birational is given in \cite[Proposition~4.5]{Reg09}. 

\begin{corollary}
\label{th:bijection-stable-points}
Let $f \colon Y \to X$ be a proper birational morphism between schemes
of finite type over a perfect field. Then the induced map $f_\infty \colon
Y_\infty \to X_\infty$ induces a bijection from 
the set of constructible points of $Y_\infty$ that are not contained in $(\Sing Y)_\infty$
and the set of constructible points of $X_\infty$ that are not contained in $(\Sing X)_\infty$. 
In particular, if $X$ and $Y$ are varieties, then $f_\infty$ induces a bijection
\[
	\{ \text{stable points $\beta\in Y_\infty$} \}
	\xrightarrow{1-1}
	\{ \text{stable points $\alpha \in X_\infty$} \}.
\]
\end{corollary}

\begin{proof}
By \cref{th:bijection-finite-embdim,th:char-finite-cylcodim}.
\end{proof}

Notice, by contrast, that the image $f_\infty(C)$ of a constructible set $C
\subset Y_\infty$ needs not be constructible in $X_\infty$. This is shown in
the next example.

\begin{example}
\label{eg:image-not-constr}
Let $f \colon Y \to X$ be the blow-up of a smooth closed point $x \in X$ of a
variety of dimension at least two, $E \subset Y$ the exceptional divisor, and $y
\in E$ a closed point. The set $C \subset Y_\infty$ of arcs with positive order
of contact with $E$ at points in $E \setminus \{y\}$ is constructible,  but its
image $f_\infty(C) \subset X_\infty$ is not constructible, since it is equal to
$W \setminus \bigcup_{i \ge 1} Z_i$ where $W$ is the set of arcs through $x$
and $Z_i$ is the set of arcs with order $i$ at $x$ and principal tangent
direction equal to $y$. 
\end{example}

\section{Maximal divisorial arcs}

In this section we study arcs that are naturally associated with divisorial
valuations. Let $X$ be a reduced scheme of finite type over a perfect field $k$. 

\begin{definition}
A \emph{valuation} on $X$ is intended to be a $k$-trivial valuation of the
function field of one of the irreducible components of $X$ with center in $X$.
A \emph{divisorial valuation} on $X$ is a valuation $v$ of the form $v =
q\ord_E$ where $q$ is a positive integer and $E$ is a prime divisor on a normal
scheme $Y$ with a morphism $f \colon Y \to X$ that is birational over a union
of irreducible components of $X$. For a divisorial valuation $v$, the number
$\widehat k_v(X) := v(\Jac_f)$ depends only on $v$ (not on the particular map
$f$), and is called the \emph{Mather discrepancy} of $v$ over $X$. When $v =
\ord_E$ (so $q=1$), we write $\widehat k_E(X)$.
\end{definition}

If we denote by $\mathcal K_X$ the sheaf of rational functions of $X$ in the
sense of \cite{Kle79}, that is, $\mathcal K_X = \bigoplus_\eta \mathcal
O_{X,\eta}$ where $\eta$ ranges among the generic points of the irreducible
components of $X$, then a valuation of $X$ can be thought as a function $v
\colon \mathcal K_X \to (-\infty,\infty]$ which restricts to a Krull valuation
on one of the summands $\mathcal O_{X,\eta}$ and is constant equal to $\infty$
on the other summands. If $v = q\ord_E$ is a divisorial valuation on $X$, then
$E$ is a divisor on one of the irreducible components of $Y$, and the component
of $X$ dominated by it corresponds to the summand of $\mathcal K_X$ where the
valuation is non-trivial. 

\begin{definition}
\label{def:max-div-arc}
A point $\alpha \in X_\infty$ is a \emph{maximal divisorial arc} if
$\ord_\alpha$ extends to a divisorial valuation on $X$ and $\alpha$ is maximal
among all points $\gamma \in X_\infty$ with $\ord_\gamma = \ord_\alpha$ (that
is, $\alpha$ is not the specialization of any other such point $\gamma$).
\end{definition}

In general, for an arc $\alpha \in X_\infty$, the function $\ord_\alpha$ is
only defined on $\mathcal O_{X,\alpha(0)}$. If $\alpha$ is a maximal divisorial
arc, then we write $\ord_\alpha = q\ord_E$ and think of it as a function on
$\mathcal K_X$. Note that for other arcs $\beta$ (for instance, if $\beta$ is
contained in $Y_\infty$ for a smaller dimensional scheme $Y \subset X$) there
may not be a natural way to extend $\ord_\beta$ to $\mathcal K_X$.

Let $f \colon Y \to X$ be a proper morphism from a normal scheme $Y$ that is
birational over a union of components. Let $E \subset Y$ be a prime divisor,
and let $E^\circ \subset E$ be the open set where both $Y$ and $E$ are smooth
and none of the other components of the exceptional locus of $f$ intersect $E$.
For any positive integer $q$, consider the \emph{contact locus}
\[
\Cont^{\ge q}(E^\circ,Y) \subset Y_\infty,
\]
which is defined to be the set of arcs in $Y$ with order of contact at least
$q$ with $E$ at a point in $E^\circ$. Since $Y$ is smooth along $E^\circ$, the
truncations $Y_m \to Y_n$ are affine bundles over an open set containing
$E^\circ$, and this implies that $\Cont^{\ge q}(E^\circ,Y)$ is
irreducible. 

The following property is well-known to experts; we include a proof for the 
convenience of the reader.

\begin{lemma}
\label{th:max-div-arc}
With the above notation, the image under $f_\infty \colon Y_\infty \to
X_\infty$ of the generic point of $\Cont^{\ge q}(E^\circ,Y)$ is a maximal
divisorial arc on $X$, and any such arc arises in this way. 
\end{lemma}

\begin{proof}
Let $\beta$ be the generic point of $\Cont^{\ge q}(E^\circ,Y)$ and $\alpha =
f_\infty(\beta)$. It is elementary to see that $\ord_\beta = q\ord_E$. By the
definition of $f_\infty$, we have $\ord_\alpha = \ord_\beta$, and therefore
$\ord_\alpha = q\ord_E$.

We may assume without loss of generality that $f$ is dominant (this is not
essential, but it makes the wording of the proof more clear). If $\gamma \in X$
is any arc with $\ord_\gamma = \ord_\alpha$, then $\gamma$ cannot be fully
contained in the indeterminacy locus of $f^{-1}$, and therefore it lifts to an
arc $\widetilde\gamma$ on $Y$ by the valuative criterion of properness. Since
$\ord_{\widetilde\gamma} = q\ord_E$, we see that $\widetilde\gamma$ must
dominate the generic point of $E$ and hence lie in $\Cont^{\ge q}(E^\circ,Y)$.
It follows that $\gamma$ is a specialization of $\alpha$, and therefore
$\alpha$ is a maximal divisorial arc. This argument also shows that any maximal
divisorial arc arises in this way. 
\end{proof}

\begin{theorem}
\label{th:max-div-arc-embdim}
Let $X$ be a reduced scheme of finite type over a perfect field. For every
divisorial valuation $q\ord_E$ on $X$ there exists a unique maximal divisorial
arc $\alpha \in X_\infty$ with $\ord_\alpha = q\ord_E$. Moreover:
\[
\embdim(\mathcal O_{X_\infty,\alpha}) = q\big(\widehat k_E(X)+1\big).
\]
\end{theorem}

\begin{proof}
The first assertion is well-known and can be viewed
as a direct consequence of \cref{th:max-div-arc}. The
formula for the embedding dimension follows from \cref{th:BTR}, after we notice
that if $\beta$ is the generic point of $\Cont^{\ge q}(E^\circ,Y)$ then
$\embdim(\mathcal O_{Y_\infty,\beta}) = q$, which is an easy computation
given that $Y$ is smooth. 
\end{proof}

From \cref{th:char-finite-cylcodim}, we recover the following fact about
maximal divisorial arcs due to \cite{ELM04}
when $X$ is a smooth variety and \cite{dFEI08,Reg09} for arbitrary varieties.

\begin{corollary}
\label{th:max-div-arcs-stable}
Let $X$ be a reduced scheme of finite type. Then every maximal divisorial arc
$\alpha \in X_\infty$ is a constructible point and is not contained in $(\Sing X)_\infty$.
\end{corollary}

\begin{proof}
By \cref{th:max-div-arc-embdim}, the local ring $\mathcal O_{X_\infty,\alpha}$
has finite embedding dimension, and hence $\alpha$ is a constructible point not contained in $(\Sing X)_\infty$
by \cref{th:char-finite-cylcodim}.
\end{proof}

Since by \cref{th:embdim=cylcodim} if $X$ is a variety then we have $\jetcodim(\alpha) =
\embdim(\mathcal O_{X_\infty,\alpha})$, we obtain the next corollary
which recovers the formula in \cite[Theorem~3.8]{dFEI08}.

\begin{corollary}
\label{th:max-div-arc-cylcodim}
With the same assumptions as in \cref{th:max-div-arc-embdim}, 
if $X$ is a variety then we have
\[
\jetcodim(\alpha, X_\infty) = q\big(\widehat k_E(X)+1\big).
\]
\end{corollary}

The following related result has been recently proved, by different methods, 
by Mourtada and Reguera.

\begin{theorem}[\cite{Reg,MR}]
\label{th:MR}
Let $X$ be a variety defined over a field of characteristic zero
and $\alpha \in X_\infty$ is a maximal divisorial point
corresponding to a valuation $q\ord_E$. Then 
\[
	\embdim(\widehat{\mathcal O_{X_\infty,\alpha}}) 
	= \embdim(\mathcal O_{(X_\infty)_\red,\alpha}) 
	= q(\widehat k_E(X) + 1).
\]
\end{theorem}

\begin{remark}
\label{r:MR-Reg09}
The proof of \cref{th:max-div-arc-embdim} does not use the fact that
maximal divisorial arcs are stable points, which is here deduced as a corollary
(see \cref{th:max-div-arcs-stable}). Granting this well-known fact from the 
start, if $X$ is a variety over a field of characteristic zero then 
one can also deduce \cref{th:max-div-arc-embdim}
from \cref{th:MR}. Indeed, under these assumptions, 
if $\alpha \in X_\infty$
is a stable point and we denote by $I \subset \mathcal O_{X_\infty,\alpha}$
and $\bar I \subset \mathcal O_{(X_\infty)_\red,\alpha}$ the respective
maximal ideals, 
then using the isomorphism 
$\widehat{\mathcal O_{X_\infty,\alpha}} \simeq \widehat{\mathcal O_{(X_\infty)_\red,\alpha}}$
proved in \cite[Theorem~3.13]{Reg09} and the fact that the
$\bar I$-adic topology of $\mathcal O_{(X_\infty)_\red,\alpha}$ is separated
by \cite[Corollary~4.3]{Reg09}, 
one deduces that the natural surjection $I/I^2 \to \bar I/\bar I^2$
is an isomorphism and hence 
\[
\embdim(\mathcal O_{X_\infty,\alpha}) = \embdim(\mathcal O_{(X_\infty)_\red,\alpha}).
\]
\end{remark}

\Cref{th:max-div-arc-embdim} can be used to control Mather discrepancies. For
example, the following result is an immediate consequence of
\cref{th:embdim>=dimX}.

\begin{corollary}
\label{th:bound-on-^kE}
Let $X$ be a reduced and equidimensional scheme of finite type over a perfect
field, and consider a prime divisor $E$ over $X$ whose center in $X$ is a
closed point. Then
\[
    \widehat k_E(X) + 1 \geq \dim(X).
\]
\end{corollary}

In fact, using \cref{th:lambda} (with $n=0$ and $n=1$) and
\cref{r:equalities}, it is not hard to see that 
if equality holds in this formula then the valuation $\ord_E$ 
has center of codimension 1 in the normalized blow-up of the maximal ideal
$\mathfrak m \subset \mathcal O_{X,x}$ at $x$, and $\ord_E(\mathfrak m) = 1$.

These facts should be compared with the following result of Ishii.
In the statement of the theorem,
$\widehat{\operatorname{mld}}_x(X)$ denotes the \emph{minimal Mather log
discrepancy} of $X$ at the point $x$, which is defined as the infimum of the
Mather log discrepancies $\widehat k_E(X)+1$ as $E$ ranges among all divisors
$E$ over $X$ with center $x$. 

\begin{theorem}[\protect{\cite[Theorem~1.1]{Ish13}}]
\label{th:bound-on-^mld}
Let $X$ be a variety over a perfect field, and $x \in X$ a closed point. 
Then 
\[
\widehat{\operatorname{mld}}_x(X) \geq \dim(X)
\]
and equality holds if and only if the normalized blow-up of the maximal ideal
$\mathfrak m \subset \mathcal O_{X,x}$ at $x$ extracts a divisor $E$ over $X$
such that $\ord_E(\mathfrak m) = 1$.
\end{theorem}

Mather log discrepancies are closely related to the usual log discrepancies, 
which are defined on $\mathbb Q$-Gorenstein varieties. 
Minimal log discrepancies are conjectured to be bounded 
above by the dimension of the variety and to characterize smooth points \cite{Sho02}. 
The above result of Ishii shows the different behavior of minimal Mather log discrepancies, 
and has useful applications in connection to Shokurov's conjecture and the 
study of isolated singularities with simple links \cite{dFT}. 

\Cref{th:bound-on-^kE} immediately implies the first statement of 
\cref{th:bound-on-^mld}. Alternatively, the full result can also be obtained by
analyzing the behavior of Mather discrepancies under 
general linear projections, in the spirit of \cite[Proposition~2.4]{dFM15};
the argument is essentially contained in the proof of \cite[Proposition~4.6]{dF17}.


\vfill
\end{document}